\documentclass[10pt,a4paper,reqno]{amsart}

\usepackage[utf8]{inputenc}
\usepackage[T1]{fontenc}
\usepackage{lmodern}

\usepackage[english]{babel}

\usepackage{amsmath}
\usepackage{amsfonts}
\usepackage{amssymb}
\usepackage{amsthm}

\usepackage{mathdots}

\usepackage{hyperref}

\theoremstyle{plain}
	\newtheorem{theorem}{Theorem}
	\newtheorem{proposition}[theorem]{Proposition}
	\newtheorem{lemma}[theorem]{Lemma}
\theoremstyle{definition}
	\newtheorem{definition}[theorem]{Definition}
	\newtheorem{notation}[theorem]{Notation}
	\newtheorem{example}[theorem]{Example}
\theoremstyle{remark}
	\newtheorem{remark}[theorem]{Remark}

\begin{document}

\title[Involutions on graded-division simple real algebras]{Classification of involutions on graded-division simple real algebras}

\author{Yuri Bahturin}
\address{Department of Mathematics and Statistics,
Memorial University of Newfoundland,
St. John's, NL, A1C5S7, Canada.}
\email{bahturin@mun.ca}

\author{Mikhail Kochetov}
\address{Department of Mathematics and Statistics,
Memorial University of Newfoundland,
St. John's, NL, A1C5S7, Canada.}
\email{mikhail@mun.ca}

\author{Adri\'an Rodrigo-Escudero}
\address{Departamento de Matem\'aticas e
Instituto Uni\-ver\-si\-ta\-rio de Matem\'aticas y Apli\-ca\-cio\-nes,
Universidad de Zaragoza, 50009 Zaragoza, Spain.}
\email{adrian.rodrigo.escudero@gmail.com}

\date{12 February 2018}
\subjclass[2010]{Primary 16W50, 16W10; secondary 16K20, 16S35.}
\keywords{Graded algebra; involution;
division grading; simple real algebra; classification.}
\thanks{\texttt{https://doi.org/10.1016/j.laa.2018.01.040}
This is the accepted manuscript
(accommodating the referee's suggestions)
of the article published in Linear Algebra and its Applications.}

\begin{abstract}
We classify,
up to isomorphism and up to equivalence,
involutions on
graded-division finite-dimensional
simple real (associative) algebras,
when the grading group is abelian.
\end{abstract}

\maketitle

\tableofcontents

\section{Introduction}\label{sect:Intr}

The study of gradings on various algebras has recently become
an active research field --- see the monograph \cite{EK-2013} 
and the references therein for an overview of this topic.
One of the milestone results in that monograph (following \cite{BZ-2007,BK-2010,Elduque-2010})
is the classification of gradings
on classical simple Lie algebras
over algebraically closed fields
of characteristic different from $2$.
It was achieved by first reducing the problem to the classification of
gradings on finite-dimensional simple associative algebras with involution 
(or, more generally, an antiautomorphism).

This was the main reason to write this article:
ultimately, we want to classify gradings on real Lie algebras,
and the first step in our approach is to study
involutions on graded-division real associative algebras.
In fact, we have already finished the classification of
gradings on classical central simple real Lie algebras
(except those of type $D_4$).
The results are to appear in a separate article
(see preprint \cite{BKR-2017}),
in which some of the arguments rely on this paper.
On the other hand, the classification of involutions (and related objects)
may be of independent interest.

Involutions on graded-division
finite-dimensional simple complex algebras
are classified in \cite[Propositions 2.51 and 2.53]{EK-2013}
(see also \cite{BZ-2006}).
In this paper we solve the real case.
As a prerequisite,
we need to know the classification of
division gradings on
finite-dimensional simple real algebras
(without involution).
This classification has been done in \cite{Rodrigo-2016},
both up to isomorphism and up to equivalence,
and independently in \cite{BZ-2016-a},
up to equivalence
(but note that one of the equivalence classes was overlooked).
A classification up to equivalence has been obtained
in \cite{BZ-2016-b} without assuming simplicity.

\medskip

The main objective of this work is to classify,
up to isomorphism and up to equivalence,
involutions on
graded-division
simple real associative algebras of finite dimension,
when the grading group is abelian.
We consider only abelian grading groups here
because of our intended applications: the support of
a grading on a simple Lie algebra
always generates an abelian subgroup of the grading group
(see for example \cite[Proposition 1.12]{EK-2013}).
Our main classification results are achieved in Sections
\ref{sect:Dim1}, \ref{sect:Dim2NonComplex},
\ref{sect:Dim2Complex} and \ref{sect:Dim4}.

\medskip

The paper is structured as follows.
We have collected the properties
that characterize involutions on
fi\-nite-di\-men\-sional simple real algebras
in Section \ref{sect:BackInv}.
Other preliminaries, such as the definitions of
isomorphism, equivalence and division grading,
can be found in Section \ref{sect:BackGrad},
together with the rest of terminology related to gradings
that we use in the paper.
Our main classification results are presented
in terms of quadratic forms on certain abelian groups 
and a similar kind of maps (which we call ``nice maps'').
These objects are introduced in Section~\ref{sect:QF}.

All homogeneous components of finite-dimensional
graded-division real algebras have the same dimension,
which can be $1$, $2$ or $4$, according to the identity component being 
the field of real numbers $\mathbb{R}$, the field of complex numbers $\mathbb{C}$ or 
the division algebra of quaternions $\mathbb{H}$.
In the case of dimension $2$, the identity component may or may not be contained in the 
center of the algebra.
Consequently, our classification results are arranged into four sections.
In Section \ref{sect:Dim1}, we classify
involutions on graded-division algebras
whose homogeneous components have dimension $1$.
In Section \ref{sect:Dim2Complex}, we consider
the case of dimension $2$ where the identity component is contained in the center, 
or, equivalently, the center is $\mathbb{C}$ with the trivial grading;
in this situation the algebra can be regarded
as a graded algebra over $\mathbb{C}$.
In Section \ref{sect:Dim2NonComplex},
we also study the case of dimension $2$,
but the identity component is not contained in the center.
Finally, the case of dimension $4$ is reduced
to the case of dimension $1$
thanks to the Double Centralizer Theorem,
as stated in Section \ref{sect:Dim4}.
Note that these four sections
are written as if they were very long theorems;
we have made an effort
to compile the classification to serve as a reference.

Section \ref{sect:GradQF} is written in the same style,
that is, as if it were a very long theorem, 
but its motivation is different.
Instead of classifying involutions,
we classify division gradings.
Moreover, the underlying algebra is not necessarily simple here.
The main goal of this section
is to classify all quadratic forms
that will appear in the following sections
and, in particular, establish their existence.
Thus, the logic of this section has 
the opposite direction as compared to the rest of the text.

As mentioned above,
we use the results of this paper
to classify gradings on classical real Lie algebras
in \cite{BKR-2017}.
There, in the case of outer gradings
on special linear Lie algebras (which belong to series $A$),
we have to deal with associative algebras
that are not simple, but simple as algebras with involution.
So, in Section \ref{sect:Semisimple} of this paper,
we extend a part of the results of the previous sections
to algebras whose center is isomorphic to $\mathbb{R}\times\mathbb{R}$.

Finally, in Section \ref{sect:Distinguished}, we discuss involutions with
special properties, which we call ``distinguished involutions''.
We use them in our preprint \cite{BKR-2017},
but they may also be of independent interest. For example,
in the situation of Section \ref{sect:Dim2Complex}, they allow us
to construct a special basis for a part of the graded-division algebra.

\section{Background on involutions}\label{sect:BackInv}

In this section we review
the basic properties of involutions on
fi\-nite-di\-men\-sional simple real algebras.
We will use \cite{KMRT-1998} as a reference.

An \emph{antiautomorphism} of an algebra $\mathcal{D}$ is
a map $ \varphi : \mathcal{D} \rightarrow \mathcal{D} $
which is an isomorphism of vector spaces and such that
$ \varphi(xy) = \varphi(y) \varphi(x) $ for all $ x,y \in \mathcal{D} $.
If it also satisfies $ \varphi^2(x) = x $ for all $ x \in \mathcal{D} $,
$\varphi$ is called an \emph{involution}.

Let $\varphi$ be an involution on a real algebra $\mathcal{D}$.
The center $Z(\mathcal{D})$ is preserved under $\varphi$,
so either the restriction of $\varphi$ to $Z(\mathcal{D})$ is the identity
and the involution is said to be \emph{of the first kind},
or this restriction has order $2$
and the involution is said to be \emph{of the second kind}.

Let $\mathbb{F}$ be either $\mathbb{R}$ or $\mathbb{C}$,
and let $V$ be an $\mathbb{F}$-vector space of dimension $n$.
An $\mathbb{F}$-bilinear form $ b : V \times V \rightarrow \mathbb{F} $ is called \emph{nonsingular} (or \emph{nondegenerate})
if the only element $ x \in V $ such that $ b(x,y) = 0 $ for all $ y \in V $ is $ x = 0 $.
The following is well known (see \cite[p.~1]{KMRT-1998}).
First, given one such $b$, there exists a unique map
$ \sigma_b : \mathrm{End}_{ \mathbb{F} }(V) \rightarrow \mathrm{End}_{ \mathbb{F} }(V) $
that satisfies the equation 
\[ b ( x , f(y) ) = b ( \sigma_b(f)(x) , y ) \]
for all $ x,y \in V $ and $ f \in \mathrm{End}_{ \mathbb{F} }(V) $.
Second, the map $ b \mapsto \sigma_b $ induces a bijective correspondence between
the classes of nonsingular $\mathbb{F}$-bilinear forms on $V$
that are either symmetric or skew-symmetric, up to a factor in $\mathbb{F}^{\times}$,
and involutions (of the first kind in the case $ \mathbb{F} = \mathbb{C} $)
on $ \mathrm{End}_{ \mathbb{F} }(V) $ ($ \cong M_n(\mathbb{F}) $).
The involutions that are adjoint to symmetric bilinear forms are called \emph{orthogonal},
while those that are adjoint to skew-symmetric bilinear forms are called \emph{symplectic}.

Let $\varphi$ be an orthogonal involution on $M_n(\mathbb{R})$,
and take a nonsingular symmetric bilinear form $b$ on a real vector space $V$
such that $\varphi$ corresponds to $\sigma_b$
via some isomorphism $ M_n(\mathbb{R}) \cong \mathrm{End}_{ \mathbb{R} }(V) $.
The number $m_+$ (respectively $m_-$) of positive (respectively negative) entries in a diagonalization of $b$
does not depend on the choice of the orthogonal basis.
Therefore, $ \vert m_+ - m_- \vert $ is an invariant of $\varphi$,
called its \emph{signature}.

An involution $\varphi$ on $M_n(\mathbb{H})$ is called \emph{orthogonal} or \emph{symplectic}
if so is its complexification $ \varphi \otimes_{\mathbb{R}} \mathrm{id}_{ \mathbb{C} } $.
We will use the following characterization (\cite[Proposition 2.6]{KMRT-1998}).
Let $\mathcal{D}$ be a finite-dimensional simple real algebra,
and let $\varphi$ be an involution on $\mathcal{D}$
(of the first kind if $ \mathcal{D} \cong M_n(\mathbb{C}) $);
then $\varphi$ is orthogonal if and only if the dimension of
$ \{ x \in \mathcal{D} \mid \varphi(x) = +x \} $
is greater than the dimension of
$ \{ x \in \mathcal{D} \mid \varphi(x) = -x \} $,
while it is symplectic if and only if it is smaller.

Let $\mathbb{D}$ be either $\mathbb{H}$ or $\mathbb{C}$,
let $V$ be a right $\mathbb{D}$-vector space of dimension $n$,
and denote by $\overline{x}$ the conjugate of $x$ in $\mathbb{D}$.
A \emph{hermitian form} on $V$ is an $\mathbb{R}$-bilinear map $ h : V \times V \rightarrow \mathbb{D} $
such that, for all $ x,y \in V $ and $ a,b \in \mathbb{D} $, we have:
(1) $ h(xa,yb) = \overline{a} h(x,y) b $ and 
(2) $ h(y,x) = \overline{h(x,y)} $.
The form is called \emph{skew-hermitian} if condition (2) is replaced by:
(2') $ h(y,x) = - \overline{h(x,y)} $.
Thus, these forms are \emph{sesquilinear}: linear in the second variable and semilinear in the first.
If we take $\mathbb{D}=\mathbb{R}$ (with $\overline{x}=x$) then we recover the definitions of symmetric and 
skew-symmetric forms.

A hermitian or skew-hermitian form $h$ is called \emph{nonsingular}
if the only element $ x \in V $ such that $ h(x,y) = 0 $ for all $ y \in V $ is $ x = 0 $.
It is well known (see \cite[Proposition 4.1]{KMRT-1998}) that,
given one such $h$, there exists a unique map
$ \sigma_h : \mathrm{End}_{ \mathbb{D} }(V) \rightarrow \mathrm{End}_{ \mathbb{D} }(V) $
that satisfies the equation 
\begin{equation}\label{eq:inv_by_h} 
h ( x , f(y) ) = h ( \sigma_h(f)(x) , y ) 
\end{equation}
for all $ x,y \in V $ and $ f \in \mathrm{End}_{ \mathbb{D} }(V) $.
Also, by \cite[Theorem 4.2]{KMRT-1998}, we have the following.
\begin{itemize}
\item In the case $ \mathbb{D} = \mathbb{H} $,
the map $ h \mapsto \sigma_h $ defines a bijective correspondence between
the classes of nonsingular hermitian (respectively skew-hermitian) forms on $V$, up to a factor in $\mathbb{R}^{\times}$,
and symplectic (respectively orthogonal) involutions on $ \mathrm{End}_{ \mathbb{H} }(V) $ ($ \cong M_n(\mathbb{H}) $).
\item In the case $ \mathbb{D} = \mathbb{C} $,
the map $ h \mapsto \sigma_h $ defines a bijective correspondence between
the classes of nonsingular hermitian forms on $V$, up to a factor in $\mathbb{R}^{\times}$,
and involutions of the second kind on $ \mathrm{End}_{ \mathbb{C} }(V) $ ($ \cong M_n(\mathbb{C}) $).
\end{itemize}

For a symplectic involution on $M_n(\mathbb{H})$ or an involution of the second kind on $M_n(\mathbb{C})$,
we define, in the same way as in the case of orthogonal involutions on $M_n(\mathbb{R})$,
the \emph{signature} as the absolute value of
the difference between the number of positive and negative entries in
any diagonalization of any adjoint hermitian form.

Finally, let us also state a couple of lemmas
for future reference.

\begin{lemma}\label{lem:SignReal}
Let $\varphi_1$ be an orthogonal
involution on $M_{n_1}(\mathbb{R})$.
Let $\mathbb{D}$ be $\mathbb{R}$
(respectively $\mathbb{H}$, $\mathbb{C}$),
and let $\varphi_2$ be an orthogonal
(respectively symplectic, second kind)
involution on $M_{n_2}(\mathbb{D})$.
Then $ \varphi_1 \otimes_{\mathbb{R}} \varphi_2 $ is an orthogonal
(respectively symplectic, second kind) involution on
$ M_{n_1}(\mathbb{R}) \otimes_{\mathbb{R}} M_{n_2}(\mathbb{D}) $,
and its signature is the product of the signatures of
$\varphi_1$ and $\varphi_2$.
\end{lemma}

\begin{proof}
Assume that $\varphi_1$ is adjoint to the bilinear form
$ b_1 : V_1 \times V_1 \rightarrow \mathbb{R} $
and $\varphi_2$ is adjoint to the hermitian form
$ h_2 : V_2 \times V_2 \rightarrow \mathbb{D} $.
Note that we have the natural isomorphism of real algebras:
\[
\mathrm{End}_{ \mathbb{R} }(V_1)
\otimes_{\mathbb{R}}
\mathrm{End}_{ \mathbb{D} }(V_2)
\cong
\mathrm{End}_{ \mathbb{D} }( V_1 \otimes_{\mathbb{R}} V_2 ).
\]
Through these identifications,
$ b_1 \otimes_{\mathbb{R}} h_2 $ is a hermitian form on
$ V_1 \otimes_{\mathbb{R}} V_2 $ adjoint to
$ \varphi_1 \otimes_{\mathbb{R}} \varphi_2 $.
Picking orthogonal bases in $V_1$ and $V_2$, we reduce the proof
to a straightforward combinatorial fact.
\end{proof}

\begin{lemma}\label{lem:SignCompl}
Let $\varphi_1$ and $\varphi_2$ be second kind involutions
on $M_{n_1}(\mathbb{C})$ and $M_{n_2}(\mathbb{C})$.
Then there is a unique second kind involution on
$ M_{n_1}(\mathbb{C}) \otimes_{\mathbb{C}} M_{n_2}(\mathbb{C}) $
that sends $ X_1 \otimes_{\mathbb{C}} X_2 $ to
$ \varphi_1(X_1) \otimes_{\mathbb{C}} \varphi_2(X_2) $;
we denote this map by $ \varphi_1 \otimes_{\mathbb{C}} \varphi_2 $.
Moreover, its signature is the product of
the signatures of $\varphi_1$ and $\varphi_2$.
\end{lemma}

\begin{proof}
Let us just recall the well known construction of
$ \varphi_1 \otimes_{\mathbb{C}} \varphi_2 $,
because the rest of the proof is analogous to
the proof of Lemma \ref{lem:SignReal}.
We can consider the $\mathbb{C}$-vector space
$ \overline{ M_{n_i}(\mathbb{C}) } $,
which has the same underlying abelian group as
$M_{n_i}(\mathbb{C})$,
but a twisted scalar multiplication $*$ given by
$ \alpha * X := \overline{\alpha} X $.
If we denote by $\overline{\varphi_i}$ the map $\varphi_i$
viewed as a map from $M_{n_i}(\mathbb{C})$ to
$ \overline{ M_{n_i}(\mathbb{C}) } $,
then $\overline{\varphi_i}$ is $\mathbb{C}$-linear.
Therefore, we have the $\mathbb{C}$-linear map:
\[
\overline{\varphi_1}
\otimes_{\mathbb{C}}
\overline{\varphi_2}
:
M_{n_1}(\mathbb{C})
\otimes_{\mathbb{C}}
M_{n_2}(\mathbb{C})
\longrightarrow
\overline{ M_{n_1}(\mathbb{C}) }
\otimes_{\mathbb{C}}
\overline{ M_{n_2}(\mathbb{C}) }.
\]
On the other hand,
we have a natural $\mathbb{C}$-linear isomorphism:
\[
\overline{ M_{n_1}(\mathbb{C}) }
\otimes_{\mathbb{C}}
\overline{ M_{n_2}(\mathbb{C}) }
\longrightarrow
\overline{
	M_{n_1}(\mathbb{C})
	\otimes_{\mathbb{C}}
	M_{n_2}(\mathbb{C})
	}.
\]
Finally, $ \varphi_1 \otimes_{\mathbb{C}} \varphi_2 $ is the $\mathbb{C}$-semilinear map
corresponding to the composition of the two maps above, and it sends
$ X_1 \otimes_{\mathbb{C}} X_2 $ to
$ \varphi_1(X_1) \otimes_{\mathbb{C}} \varphi_2(X_2) $.
\end{proof}

\section{Background on gradings}\label{sect:BackGrad}

In this section we review, following \cite{EK-2013}, the basic definitions and properties of gradings that will be used in the rest of the paper.
Here we only deal with associative algebras.

\begin{definition}
Let $\mathcal{D}$ be an algebra over a field $\mathbb{F}$, and let $G$ be a group.
A \emph{$G$-grading} $\Gamma$ on $\mathcal{D}$ is a decomposition of $\mathcal{D}$ into a direct sum of subspaces indexed by $G$,
\[ \Gamma : \mathcal{D} = \bigoplus_{ g \in G } \mathcal{D}_g, \]
such that, for all $ g,h \in G$, we have
\[ \mathcal{D}_g \mathcal{D}_h \subseteq \mathcal{D}_{gh}. \]
If such a decomposition is fixed, we refer to $\mathcal{D}$ as a \emph{$G$-graded algebra}.
The \emph{support} of $\Gamma$ (or of $\mathcal{D}$) is the set $ \mathrm{supp} (\Gamma) := \{ g \in G \mid \mathcal{D}_g \neq 0 \} $.
If $ X \in \mathcal{D}_g $, then we say that $X$ is \emph{homogeneous of degree $g$}, and we write $ \deg (X) = g $.
The subspace $\mathcal{D}_g$ is called the \emph{homogeneous component of degree $g$}.
\end{definition}

Note that, if $\mathcal{D}$ is a $G$-graded algebra and $\mathcal{D}'$ is an $H$-graded algebra, then the tensor product $ \mathcal{D} \otimes \mathcal{D}' $ has a natural $ G \times H $-grading given by $ ( \mathcal{D} \otimes \mathcal{D}' )_{(g,h)} = \mathcal{D}_g \otimes \mathcal{D}'_h $, for all $ g \in G $, $ h \in H $.
This will be called the \emph{product grading}.

A subspace $\mathcal{F}$ (in particular, a subalgebra or an ideal) of a $G$-graded algebra $\mathcal{D}$ is said to be \emph{graded} 
if $ \mathcal{F} = \bigoplus_{ g \in G } ( \mathcal{D}_g \cap \mathcal{F} ) $.

There are two natural ways to define an equivalence relation on group gradings, depending on whether the grading group plays a secondary role or is a part of the definition.

\begin{definition}\label{def:EquivGrad}
Let $\Gamma$ be a $G$-grading on the algebra $\mathcal{D}$ and let $\Gamma'$ be an $H$-grading on the algebra $\mathcal{D}'$.
We say that $\Gamma$ and $\Gamma'$ are \emph{equivalent} if there exist an isomorphism of algebras $ \psi : \mathcal{D} \rightarrow \mathcal{D}' $ and a bijection $ \alpha : \mathrm{supp} (\Gamma) \rightarrow \mathrm{supp} (\Gamma') $ such that $ \psi (\mathcal{D}_t) = \mathcal{D}'_{\alpha(t)} $ for all $ t \in \mathrm{supp} (\Gamma) $.
\end{definition}

\begin{definition}
Let $\Gamma$ and $\Gamma'$ be $G$-gradings on the algebras $\mathcal{D}$ and $\mathcal{D}'$, respectively.
We say that $\Gamma$ and $\Gamma'$ are \emph{isomorphic} if there exists an isomorphism of algebras $ \psi : \mathcal{D} \rightarrow \mathcal{D}' $ such that $ \psi (\mathcal{D}_g) = \mathcal{D}'_g $ for all $ g \in G $.
\end{definition}

\begin{definition}
Given gradings on the same algebra, 
$ \Gamma : \mathcal{D} = \bigoplus_{ g \in G } \mathcal{D}_g $ and $ \Gamma' : \mathcal{D} = \bigoplus_{ h \in H } \mathcal{D}'_h $, 
we say that $\Gamma'$ is a \emph{coarsening} of $\Gamma$, or that $\Gamma$ is a \emph{refinement} of $\Gamma'$, 
if, for any $ g \in G $, there exists $ h \in H $ such that $ \mathcal{D}_g \subseteq \mathcal{D}'_h $.
If, for some $ g \in G $, this inclusion is strict, then we will speak of a \emph{proper} refinement or coarsening.
A grading is said to be \emph{fine} if it does not admit a proper refinement.
\end{definition}

\begin{definition}
A graded algebra is said to be a \emph{graded division algebra} if it is unital and every nonzero homogeneous element has an inverse.
In this case, the grading will be called a \emph{division grading}.
\end{definition}

If $\mathcal{D}$ is a $G$-graded division algebra, then $ I \in \mathcal{D}_e $, where $e$ is the identity element of $G$ and $I$ the unity of $\mathcal{D}$.
Also, if $ 0 \neq X \in \mathcal{D}_g $, then $ X^{-1} \in \mathcal{D}_{g^{-1}} $.
Therefore, the support of $\mathcal{D}$ is a subgroup of $G$, since whenever $ \mathcal{D}_g \neq 0 $ and $ \mathcal{D}_h \neq 0 $, we also have $ 0 \neq \mathcal{D}_g \mathcal{D}_h \subseteq \mathcal{D}_{gh} $ and $ \mathcal{D}_{g^{-1}} \neq 0 $.
This also shows that, in the situation of Definition \ref{def:EquivGrad}, if $\Gamma$ and $\Gamma'$ are division gradings, then $ \alpha : \mathrm{supp} (\Gamma) \rightarrow \mathrm{supp} (\Gamma') $ is a homomorphism of groups.

The identity component $ \mathcal{D}_e $ of a graded division algebra $\mathcal{D}$ is a division algebra.
Also, if $ X_t \in \mathcal{D}_t $ is nonzero, then $ \mathcal{D}_t = \mathcal{D}_e X_t $. 
Therefore, all the (nonzero) homogeneous components of the grading have the same dimension.
In our case $\mathcal{D}$ will be finite-dimensional and the ground field will be $\mathbb{R}$,
so this dimension must be 1, 2 or 4 depending on whether $ \mathcal{D}_e $ is isomorphic to $ \mathbb{R} $, $ \mathbb{C} $ or $ \mathbb{H} $.

\begin{definition}
Let $\mathcal{D}$ be a $G$-graded algebra.
A map $ \varphi : \mathcal{D} \rightarrow \mathcal{D} $ is said to be an
\emph{antiautomorphism of the $G$-graded algebra $\mathcal{D}$}
if it is an isomorphism of vector spaces such that
$ \varphi(XY) = \varphi(Y) \varphi(X) $ for all $ X,Y \in \mathcal{D} $
and $ \varphi(\mathcal{D}_g) = \mathcal{D}_g $ for all $ g \in G $.
If it also satisfies $ \varphi^2(X) = X $ for all $ X \in \mathcal{D} $,
$\varphi$ is called an \emph{involution} of the $G$-graded algebra $\mathcal{D}$.
\end{definition}

\begin{definition}
Let $\Gamma$ and $\Gamma'$ be gradings on
the algebras $\mathcal{D}$ and $\mathcal{D}'$.
Let $ \varphi : \mathcal{D} \rightarrow \mathcal{D} $
and $ \varphi' : \mathcal{D}' \rightarrow \mathcal{D}' $
be antiautomorphisms of graded algebras.
We say that $(\Gamma,\varphi)$ is \emph{isomorphic}
(respectively \emph{equivalent})
to $(\Gamma',\varphi')$ if there exists an isomorphism
(respectively equivalence) of graded algebras
$ \psi : \mathcal{D} \rightarrow \mathcal{D}' $
such that $ \varphi' = \psi \varphi \psi^{-1} $.
\end{definition}

We will use the following result \cite[Lemma 3.3]{Elduque-2010}.
\begin{lemma}\label{lem:InnAut}
Let $\mathcal{D}$ be an algebra endowed with
a division grading by an abelian group $G$.
Let $ X \in \mathcal{D}^{\times} $ and consider
the corresponding inner automorphism
$ \mathrm{Int}(X) : \mathcal{D} \rightarrow \mathcal{D} $
given by $ \mathrm{Int}(X)(Y) = X Y X^{-1} $.
If $ \mathrm{Int}(X)(\mathcal{D}_g)
\subseteq \mathcal{D}_g $ for all $ g \in G $,
then there exists a nonzero homogeneous $ X_0 \in \mathcal{D} $
such that $ \mathrm{Int}(X) = \mathrm{Int}(X_0) $.
\end{lemma}

\begin{proof}
Let $ \psi := \mathrm{Int}(X) $ and note that
$ \psi(Y)X = XY $ for all $ Y \in \mathcal{D} $.
Write $ X = X_0 + \ldots + X_n $ where the $X_i$ are nonzero
homogeneous elements of pairwise different degrees $g_i$.
If $ Y \in \mathcal{D} $ is homogeneous of degree $g$,
so is $\psi(Y)$.
Since $G$ is abelian,
if we consider the terms of degree
$ g g_0 $ in the equation $ \psi(Y)X = XY $,
we get $ \psi(Y) X_0 = X_0 Y $.
But $X_0$ is invertible because it is homogeneous,
so $ \psi(Y) = X_0 Y X_0^{-1} $
for all $ Y \in \mathcal{D}_g $ and $ g \in G $.
\end{proof}

\section{Quadratic forms}\label{sect:QF}

In this section we introduce the necessary terminology concerning quadratic forms on certain abelian groups, 
mainly following \cite[Section 4]{Rodrigo-2016}.
That article established a correspondence between isomorphism classes of division gradings
and quadratic forms that are regular in the sense that usually appears in the literature.
In this paper, however, we deal with quadratic forms that satisfy less restrictive conditions of regularity.
The notation here is congruent with \cite{Rodrigo-2016},
but note that now we do not require quadratic forms and nice maps to be defined on elementary abelian $2$-groups.

\begin{definition}
Let $T$ be a finite abelian group.
In this article, by an \emph{alternating bicharacter} on $T$ we will mean a map $ \beta : T \times T \rightarrow \mathbb{R}^{\times} $ that satisfies
$ \beta(uv,w) = \beta(u,w) \beta(v,w) $, $ \beta(u,vw) = \beta(u,v) \beta(u,w) $, and $ \beta(u,u) = 1 $ for all $ u,v,w \in T $.
(It follows that $\beta$ takes values in $\{ \pm 1 \}$.)
If we have to consider an alternating bicharacter $\beta$ that takes values in $\mathbb{C}^{\times}$ instead of $\mathbb{R}^{\times}$,
we will explicitly say that $\beta$ is $\mathbb{C}$-valued.
A \emph{quadratic form} on $T$ is a map $ \mu : T \rightarrow \{ \pm 1 \} $ such that $\beta_{\mu}$ is an alternating bicharacter,
where $ \beta_{\mu} : T \times T \rightarrow \{ \pm 1 \} $,
called the \emph{polarization} of $\mu$,
is defined by
\begin{equation}\label{eq:BetaMu}
\beta_{\mu}(u,v) := \mu(uv) \mu(u)^{-1} \mu(v)^{-1}.
\end{equation}
(The inverses above have no effect, but this way the equation is more similar to the usual definition 
of quadratic forms on a vector space.)
\end{definition}

\begin{definition}\label{def:bichar_two_types}
Let $\beta$ be an alternating bicharacter on $T$, and consider its radical:
$ \mathrm{rad}(\beta) := \{ t \in T \mid \beta(u,t)=1 , \, \forall u \in T \} $.
We say that $\beta$ has \emph{type I} if $ \mathrm{rad}(\beta) = \{ e \} $,
and that it has \emph{type II} if $ \mathrm{rad}(\beta) = \{ e,f \} $ for some $f\in T$ (of order $2$).
In the latter case, as $f$ is determined by $\beta$, we denote it by $f_{\beta}$.
\end{definition}

We will say that a family $ \{ a_1 , b_1 , \ldots , a_m , b_m \}$ in $ T $
is \emph{symplectic} if $ \beta(a_i,b_i) = \beta(b_i,a_i) = -1 $ ($ i = 1 , \ldots , m $)
and the value of $\beta$ on all other pairs is $+1$.
We will say that it is a \emph{basis} if $T$ is the direct product of the subgroups
$ \langle a_1 \rangle , \langle b_1 \rangle , \ldots , \langle a_m \rangle , \langle b_m \rangle$.

The following result \cite[Proposition 9]{Rodrigo-2016} describes 
alternating bicharacters satisfying Definition \ref{def:bichar_two_types}.

\begin{proposition}\label{prop:BasisAltBich}
Let $\beta$ be an alternating bicharacter on a finite abelian group $T$.
\begin{enumerate}
	\item If $\beta$ has type I, then $ T \cong (\mathbb{Z}_2^2)^m $
	and there exists a symplectic basis $ \{ a_1 , \allowbreak b_1 , \allowbreak \ldots , \allowbreak a_m , \allowbreak b_m \} $ of $T$.
	\item If $\beta$ has type II, then either $ T \cong (\mathbb{Z}_2^2)^m \times \mathbb{Z}_2 $ or
	$ (\mathbb{Z}_2^2)^{m-1} \times ( \mathbb{Z}_2 \times \mathbb{Z}_4 ) $,
	and there exists a symplectic family $ \{ a_1 , \allowbreak b_1 , \allowbreak \ldots , \allowbreak a_m , \allowbreak b_m \} $ such that:
	\begin{enumerate}
		\item In the first case, $ \{ a_1 , b_1 , \ldots , a_m , b_m , f_{\beta} \} $ is a basis of $T$.
		\item In the second case, $ b_m^2 = f_{\beta} $ and $ \{ a_1 , b_1 , \ldots , a_m , b_m \} $ is a basis of $T$.
	\end{enumerate}
\end{enumerate}
\qed 
\end{proposition}

\begin{lemma}\label{lem:index2}
Let $\mu$ and $\eta$ be two different quadratic forms
on a finite abelian group $T$
such that $ \beta_{\mu} = \beta_{\eta} $.
Then $ \{ t \in T \mid \mu(t) = \eta(t) \} $
is a subgroup of $T$ of index $2$.
\end{lemma}

\begin{proof}
Let $ S := \{ t \in T \mid \mu(t) = \eta(t) \} $.
By Equation \eqref{eq:BetaMu}, we have $\mu(e)=1=\eta(e)$, so $e\in S$.
Also, $ u,v \in S $ implies $ uv \in S $, 
hence $S$ is a subgroup.
Since $\mu$ and $\eta$ take values in $ \{ \pm 1 \} $,
$ u,v \in T \setminus S $ implies $\mu(u)\mu(v)=\eta(u)\eta(v)$, 
hence $ uv \in S $ by Equation \eqref{eq:BetaMu}.
Thus the quotient group $ T / S $ can have only two elements.
\end{proof}

\begin{lemma}\label{lem:Ind2SubgrTyp1}
Let $\beta$ be an alternating bicharacter of type I
on a finite abelian group $T$.
Then the following map is a bijection:
\begin{align*}
T & \longrightarrow \{ S \mid S\text{ is a subgroup of }T,\:[T:S] \leq 2 \}
\\
u & \longmapsto u^{\perp} = \{ v \in T \mid \beta(u,v) = 1 \}
\end{align*}
\end{lemma}

\begin{proof}
It is enough to interpret $\beta$ as
a nondegenerate alternating bilinear form
over the field of two elements.
\end{proof}

\begin{lemma}\label{lem:Ind2SubgrTyp2}
Let $\beta$ be an alternating bicharacter of type II
on a finite abelian group $T$.
Then the following map is a bijection:
\begin{align*}
T / \langle f_{\beta} \rangle
& \longrightarrow
\{ S \mid S\text{ is a subgroup of }T,\:f_{\beta} \in S , \: [T:S] \leq 2 \}
\\
[u]
& \longmapsto
u^{\perp} = \{ v \in T \mid \beta(u,v) = 1 \}
\end{align*}
\end{lemma}

\begin{proof}
Consider the alternating bicharacter $\bar{\beta}$
on $ T / \langle f_{\beta} \rangle $ such that
$ \beta = \bar{\beta} \circ ( \pi \times \pi ) $,
where $ \pi : T \rightarrow T / \langle f_{\beta} \rangle $
is the natural projection.
Then $\bar{\beta}$ has type I and we can apply Lemma \ref{lem:Ind2SubgrTyp1}.
\end{proof}

\begin{notation}\label{nota:subgroups}
For any natural number $n$ and abelian group $T$, we define
$ T_{[n]} =  \{ t\in T \mid t^n = e \} $ and  $ T^{[n]} = \{ t^n \mid t \in T \} $.
\end{notation}

(We will primarily need the case $n=2$.)

\begin{notation}
Let $\beta$ be an alternating bicharacter of type II on $ T \cong \mathbb{Z}_2^{2m-1} \times \mathbb{Z}_4 $.
Then $T^{[2]}$ has order $2$ and we denote by $f_T$ its generator.
By Proposition \ref{prop:BasisAltBich}, $ f_T = f_{\beta} $.
We set $ \mathrm{rad}'(\beta) := \mathrm{rad} ( \beta \vert_{ T_{[2]} \times T_{[2]} } ) \setminus \mathrm{rad}(\beta) $
(which equals $ \{ a_m , a_m f_T \} $ with the notation of Proposition \ref{prop:BasisAltBich}).
\end{notation}

\begin{remark}\label{rem:fT}
If $\eta$ is a quadratic form on $ T \cong \mathbb{Z}_2^{2m-1} \times \mathbb{Z}_4 $ such that $\beta_{\eta}$ has type II,
then $ \eta(f_T) = +1 $.
Indeed, if $g$ is an element of $T$ of order $4$, then $ f_T = g^2 $,
so $ \eta(f_T) = \eta(g^2) = \eta(g)^2 \beta_{\eta}(g,g) = +1 $.
Also note that $\eta$ takes the same value on the two elements of $\mathrm{rad}'(\beta_{\eta})$,
because one is the other multiplied by $f_T$.
Finally, if $\beta$ is an alternating bicharacter of type II
on $ T \cong \mathbb{Z}_2^{2m-1} \times \mathbb{Z}_4 $,
and $\mu$ is a quadratic form defined only on $T_{[2]}$ such that
$ \beta_{\mu} = \beta \vert_{ T_{[2]} \times T_{[2]} } $
and $ \mu(f_T) = +1 $,
then there exist exactly two quadratic forms on $T$
that extend $\mu$
and whose polarization is $\beta$.
\end{remark}

\begin{proposition}\label{prop:NiceMap}
Let $T$ be a finite abelian group, $K$ a subgroup of $T$ of index $2$,
and $ \nu : T \setminus K \rightarrow \{ \pm 1 \} $ a map.
Consider the family of maps $ \mu_g : K \rightarrow \{ \pm 1 \} $
defined by $ \mu_g(k) := \nu(gk) \nu(g)^{-1} $,
as $g$ runs through $ T \setminus K $.
Then, if a member of this family is a quadratic form, so are the others,
and all have the same
polarization $\beta$.
Moreover, if $\beta$ has type II,
the value $\mu_g(f_{\beta})$ does not depend on the choice of $ g \in T \setminus K $.
\end{proposition}

\begin{proof}
Let $ g,h \in T \setminus K $, and assume that $\mu_g$ is a quadratic form.
Call $\beta$ its
polarization.
The assertions follow from the following formula ($ k \in K $):
\begin{equation}\label{eq:ProofNiceMap}
\mu_h(k) = \frac{\nu(hk)}{\nu(h)} = \frac{ \mu_g(g^{-1}hk) }{ \mu_g(g^{-1}h) } = \mu_g(k) \beta( g^{-1}h , k ).
\end{equation}
Indeed, as $\beta$ is multiplicative in $k$, $\mu_h$ is also a quadratic form with the same
polarization
as $\mu_g$.
And if $\beta$ has type II, then $ \mu_h(f_{\beta}) = \mu_g(f_{\beta}) \beta( g^{-1}h , f_{\beta} ) = \mu_g(f_{\beta}) $.
\end{proof}

\begin{definition}
In the situation of Proposition \ref{prop:NiceMap},
we say that $\nu$ is a \emph{nice map} on $ T \setminus K $,
and we denote by $\beta_{\nu}$ the common
polarization $\beta$ of the quadratic forms $\mu_g$.
If $\beta$ has type II,
we also define $ \nu(f_{\beta}) := \mu_g(f_{\beta}) $,
where $g$ is any element of $ T \setminus K $.
\end{definition}

\begin{lemma}\label{lem:NuRad}
Under the conditions of Proposition \ref{prop:NiceMap},
suppose that $ T \cong \mathbb{Z}_2^{2m-2} \times \mathbb{Z}_4 $,
$ K \cong \mathbb{Z}_2^{2m-3} \times \mathbb{Z}_4 $,
and $\beta = \beta_\nu$ has type II.
Then the set $\mu_g (\mathrm{rad}'(\beta))$,
where $g$ is any element of $ T \setminus K $ of order $2$,
does not depend on the choice of $g$.
\end{lemma}

\begin{proof}
Let $ g,h \in T \setminus K $ be elements of order $2$
and let $ a \in \mathrm{rad}'(\beta) \subseteq \mathrm{rad} ( \beta \vert_{ T_{[2]} \times T_{[2]} } ) $.
By Equation \eqref{eq:ProofNiceMap}, $ \mu_h(a) = \mu_g(a) \beta( g^{-1}h , a ) = \mu_g(a) $.
\end{proof}

\begin{notation}
In the situation of Lemma \ref{lem:NuRad},
we define $ \nu( \mathrm{rad}'(\beta) ) := \mu_g (\mathrm{rad}'(\beta))$,
where $g$ is any element of $ T \setminus K $ of order $2$.
Note that, by Remark \ref{rem:fT} (applied to $K$), $\mu_g$ takes the same value on the two elements of $\mathrm{rad}'(\beta)$,
so the set $\nu (\mathrm{rad}'(\beta))$ actually consists of one element.
\end{notation}

\begin{notation}[Arf invariant]
Let $T$ be a finite set and let $ \mu : T \rightarrow \{ \pm 1 \} $ be a map.
If the cardinality of $\mu^{-1}(+1)$ is greater than the cardinality of $\mu^{-1}(-1)$, we write $ \mathrm{Arf}(\mu) = +1 $.
If it is smaller, we write $ \mathrm{Arf}(\mu) = -1 $.
Finally, if both cardinalities are equal, $ \mathrm{Arf}(\mu) $ is not defined.
\end{notation}

\section{Division gradings and quadratic forms}\label{sect:GradQF}

As mentioned in the Introduction,
the main purpose of this section
is to classify all quadratic forms
whose
polarization
has type I or II.
We establish a correspondence with gradings
in order to prove their existence.

The following division gradings will be our building blocks.

\begin{example}\label{exam:GradDim1}

Two division gradings by the group $\mathbb{Z}_2$:
\[ 
\mathbb{C} = \mathbb{R}1 \oplus \mathbb{R}i \quad \text{and} \quad
\mathbb{R} \times \mathbb{R} = \mathbb{R}(1,1) \oplus \mathbb{R}(1,-1). 
\]
Two division gradings by the group $\mathbb{Z}_2^2$:
\[ M_2(\mathbb{R})=
\mathbb{R} \begin{pmatrix} 1 & 0 \\ 0 & 1 \end{pmatrix} \oplus
\mathbb{R} \begin{pmatrix} 0 & 1 \\ 1 & 0 \end{pmatrix} \oplus
\mathbb{R} \begin{pmatrix} -1 & 0 \\ 0 & 1 \end{pmatrix} \oplus
\mathbb{R} \begin{pmatrix} 0 & -1 \\ 1 & 0 \end{pmatrix} \]
and
\[ \mathbb{H} = \mathbb{R}1 \oplus \mathbb{R}i \oplus \mathbb{R}j \oplus \mathbb{R}k. \]
And the three division gradings by the group
$ \mathbb{Z}_2 \times \mathbb{Z}_4 =
\{ e,a ; \allowbreak b,ab ; \allowbreak
b^2,ab^2 ; \allowbreak b^3,ab^3 \} $
presented in Figure \ref{fig:DivGrad}.
\begin{figure}
\begin{align*}
M_2(\mathbb{C}) = &
	\mathbb{R} \begin{pmatrix} 1 & 0 \\ 0 & 1 \end{pmatrix}
	\oplus
	\mathbb{R} \begin{pmatrix} 0 & 1 \\ 1 & 0 \end{pmatrix}
	\oplus
\\
&
	\mathbb{R} \begin{pmatrix} e^{ \pi i / 4 } & 0 \\ 0 & -e^{ \pi i / 4 } \end{pmatrix}
	\oplus
	\mathbb{R} \begin{pmatrix} 0 & -e^{ \pi i / 4 } \\ e^{ \pi i / 4 } & 0 \end{pmatrix}
	\oplus
\\
&
	\mathbb{R} \begin{pmatrix} i & 0 \\ 0 & i \end{pmatrix}
	\oplus
	\mathbb{R} \begin{pmatrix} 0 & i \\ i & 0 \end{pmatrix}
	\oplus
\\
&
	\mathbb{R} \begin{pmatrix} e^{ 3 \pi i / 4 } & 0 \\ 0 & -e^{ 3 \pi i / 4 } \end{pmatrix}
	\oplus
	\mathbb{R} \begin{pmatrix} 0 & -e^{ 3 \pi i / 4 } \\ e^{ 3 \pi i / 4 } & 0 \end{pmatrix}
\end{align*}

\begin{align*}
M_2( \mathbb{R} \times \mathbb{R} ) = &
	\mathbb{R} \begin{pmatrix} (1,1) & (0,0) \\ (0,0) & (1,1) \end{pmatrix}
	\oplus
	\mathbb{R} \begin{pmatrix} (1,1) & (0,0) \\ (0,0) & (-1,-1) \end{pmatrix}
	\oplus
\\
&
	\mathbb{R} \begin{pmatrix} (0,0) & (1,-1) \\ (1,1) & (0,0) \end{pmatrix}
	\oplus
	\mathbb{R} \begin{pmatrix} (0,0) & (1,-1) \\ (-1,-1) & (0,0) \end{pmatrix}
	\oplus
\\
&
	\mathbb{R} \begin{pmatrix} (1,-1) & (0,0) \\ (0,0) & (1,-1) \end{pmatrix}
	\oplus
	\mathbb{R} \begin{pmatrix} (1,-1) & (0,0) \\ (0,0) & (-1,1) \end{pmatrix}
	\oplus
\\
&
	\mathbb{R} \begin{pmatrix} (0,0) & (1,1) \\ (1,-1) & (0,0) \end{pmatrix}
	\oplus
	\mathbb{R} \begin{pmatrix} (0,0) & (1,1) \\ (-1,1) & (0,0) \end{pmatrix}
\end{align*}

\begin{align*}
M_2(\mathbb{R}) \times \mathbb{H} = &
	\mathbb{R} \left( \begin{pmatrix} 1 & 0 \\ 0 & 1 \end{pmatrix} , 1 \right)
	\oplus
	\mathbb{R} \left( \begin{pmatrix} 0 & -1 \\ 1 & 0 \end{pmatrix} , i \right)
	\oplus
\\
&
	\mathbb{R} \left( \begin{pmatrix} 1 & 0 \\ 0 & -1 \end{pmatrix} , j \right)
	\oplus
	\mathbb{R} \left( \begin{pmatrix} 0 & 1 \\ 1 & 0 \end{pmatrix} , k \right)
	\oplus
\\
&
	\mathbb{R} \left( \begin{pmatrix} 1 & 0 \\ 0 & 1 \end{pmatrix} , -1 \right)
	\oplus
	\mathbb{R} \left( \begin{pmatrix} 0 & -1 \\ 1 & 0 \end{pmatrix} , -i \right)
	\oplus
\\
&
	\mathbb{R} \left( \begin{pmatrix} 1 & 0 \\ 0 & -1 \end{pmatrix} , -j \right)
	\oplus
	\mathbb{R} \left( \begin{pmatrix} 0 & 1 \\ 1 & 0 \end{pmatrix} , -k \right)
\end{align*}
\caption{Division gradings of Example \ref{exam:GradDim1}; degrees in the group $\langle a\rangle\times\langle b\rangle\cong\mathbb{Z}_2 \times \mathbb{Z}_4$ 
are assigned line-by-line in the following order: $e,a ; \allowbreak b,ab ; \allowbreak b^2,ab^2 ; \allowbreak b^3,ab^3$.}
\label{fig:DivGrad}
\end{figure}
\end{example}

Let $\mathcal{D}$ be a finite-dimensional real (associative) algebra whose center $Z(\mathcal{D})$ has dimension $1$ or $2$.
Let $G$ be an abelian group and let $\Gamma$ be a division $G$-grading on $\mathcal{D}$  
with support $T$ and homogeneous components of dimension $1$. 
Note that $\mathcal{D}$ must be unital, but we do not assume that it is a simple algebra.
By a generalization of Maschke's Theorem (see for example \cite[Corollary 10.2.5 on p.~443]{Karpilovsky-1992}),
$\mathcal{D}$ is necessarily semisimple, which implies that it is simple if $Z(\mathcal{D})$ is $\mathbb{R}$ or $\mathbb{C}$, 
and the direct product of two simple algebras if $Z(\mathcal{D})$ is $\mathbb{R}\times\mathbb{R}$. 

We claim that the graded algebra $\mathcal{D}$ is equivalent to one, and only one, tensor product on the list
below,
equipped with the product grading where
each factor is graded as in Example \ref{exam:GradDim1}.
The isomorphism classes are in bijective correspondence
with the triples $(T,\beta,\mu)$,
where $\beta$ is an alternating bicharacter on $T$ of type I or II,
and $\mu$ is a quadratic form on $T_{[2]}$ such that
$ \beta_{\mu} = \beta \vert_{ T_{[2]} \times T_{[2]} } $.
Namely, $ \beta : T \times T \rightarrow \{ \pm 1 \} $ is defined
by the commutation relations of homogeneous elements,
\begin{equation}\label{eq:CommRel}
X_u X_v = \beta (u,v) X_v X_u
\end{equation}
for all $ X_u \in \mathcal{D}_u $ and $ X_v \in \mathcal{D}_v $,
and $ \mu : T_{[2]} \rightarrow \{ \pm 1 \} $ is defined
by the signs of the squares of homogeneous elements,
\begin{equation}\label{eq:QuadrForm}
X_t^2 = \mu (t) I
\end{equation}
for all $ X_t \in \mathcal{D}_t $ such that $ X_t^4 = I $,
where $I$ is the unity of $\mathcal{D}$ and $ t \in T_{[2]} $.
Conversely, given such $(T,\beta,\mu)$,
we can construct $\mathcal{D}$ as the algebra with generators $X_u$, with $u$ ranging over a basis of $T$,
and defining relations given by Equations \eqref{eq:CommRel} for all basis elements $u$ and $v$, 
Equations \eqref{eq:QuadrForm} for all basis elements $t$ of order $2$,
and $ X_t^4 = \mu(f_T) I $ for the basis element $t$ of order $4$ (if it is present).
The grading on $\mathcal{D}$ is defined by declaring the generator $X_u$ to be of degree $u$.
Now we give the list of the equivalence classes,
and we also compile the classification up to isomorphism
to serve as a reference:

\medskip

(1-a)
$ M_n(\mathbb{R}) \cong M_2(\mathbb{R}) \otimes \ldots \otimes M_2(\mathbb{R}) $,
$ n = 2^m \geq 1 $ (if $n=1$, $ M_n(\mathbb{R}) = \mathbb{R} $ with the trivial grading).
The grading $\Gamma$ is determined up to isomorphism by $(T,\mu)$,
where $T$ is a subgroup of $G$ isomorphic to $\mathbb{Z}_2^{2m}$,
and $\mu$ is a quadratic form on $T$ such that $\beta_{\mu}$ has type I and $ \mathrm{Arf}(\mu) = +1 $.

(1-b)
$ M_{n/2}(\mathbb{H}) \cong M_2(\mathbb{R}) \otimes \ldots \otimes M_2(\mathbb{R}) \otimes \mathbb{H} $,
$ n = 2^m \geq 2 $.
The grading $\Gamma$ is determined up to isomorphism by $(T,\mu)$,
where $T$ is a subgroup of $G$ isomorphic to $\mathbb{Z}_2^{2m}$,
and $\mu$ is a quadratic form on $T$ such that $\beta_{\mu}$ has type I and $ \mathrm{Arf}(\mu) = -1 $.

(1-c)
$ M_n(\mathbb{C}) \cong M_2(\mathbb{R}) \otimes \ldots \otimes M_2(\mathbb{R}) \otimes \mathbb{C} $,
$ n = 2^m \geq 1 $.
The grading $\Gamma$ is determined up to isomorphism by $(T,\mu)$,
where $T$ is a subgroup of $G$ isomorphic to $\mathbb{Z}_2^{2m+1}$,
and $\mu$ is a quadratic form on $T$ such that $\beta := \beta_{\mu}$ has type II and $ \mu(f_{\beta}) = -1 $.

(1-d)
$ M_n(\mathbb{C}) \cong M_2(\mathbb{R}) \otimes \ldots \otimes M_2(\mathbb{R}) \otimes M_2(\mathbb{C}) $,
$ n = 2^m \geq 2 $.
The grading $\Gamma$ is determined up to isomorphism by $(T,\beta,\mu)$,
where $T$ is a subgroup of $G$ isomorphic to $ \mathbb{Z}_2^{2m-1} \times \mathbb{Z}_4 $,
$\beta$ is an alternating bicharacter on $T$ of type II,
and $\mu$ is a quadratic form on $T_{[2]}$ such that $ \beta_{\mu} = \beta \vert_{ T_{[2]} \times T_{[2]} } $ and $ \mu(f_T) = -1 $.

(1-e)
$ M_n(\mathbb{R}) \times M_n(\mathbb{R}) \cong M_2(\mathbb{R}) \otimes \ldots \otimes M_2(\mathbb{R}) \otimes [ \mathbb{R} \times \mathbb{R} ] $,
$ n = 2^m \geq 1 $.
The grading $\Gamma$ is determined up to isomorphism by $(T,\mu)$,
where $T$ is a subgroup of $G$ isomorphic to $\mathbb{Z}_2^{2m+1}$,
and $\mu$ is a quadratic form on $T$ such that $\beta := \beta_{\mu}$ has type II,
$ \mu(f_{\beta}) = +1 $ and $ \mathrm{Arf}(\mu) = +1 $.

(1-f)
$ M_{n/2}(\mathbb{H}) \times M_{n/2}(\mathbb{H}) \cong M_2(\mathbb{R}) \otimes \ldots \otimes M_2(\mathbb{R}) \otimes \mathbb{H} \otimes [ \mathbb{R} \times \mathbb{R} ] $,
$ n = 2^m \geq 2 $.
The grading $\Gamma$ is determined up to isomorphism by $(T,\mu)$,
where $T$ is a subgroup of $G$ isomorphic to $\mathbb{Z}_2^{2m+1}$,
and $\mu$ is a quadratic form on $T$ such that $\beta := \beta_{\mu}$ has type II,
$ \mu(f_{\beta}) = +1 $ and $ \mathrm{Arf}(\mu) = -1 $.

(1-g)
$ M_n(\mathbb{R}) \times M_n(\mathbb{R}) \cong M_2(\mathbb{R}) \otimes \ldots \otimes M_2(\mathbb{R}) \otimes M_2( \mathbb{R} \times \mathbb{R} ) $,
$ n = 2^m \geq 2 $.
The grading $\Gamma$ is determined up to isomorphism by $(T,\beta,\mu)$,
where $T$ is a subgroup of $G$ isomorphic to $ \mathbb{Z}_2^{2m-1} \times \mathbb{Z}_4 $,
$\beta$ is an alternating bicharacter on $T$ of type II,
and $\mu$ is a quadratic form on $T_{[2]}$ such that $ \beta_{\mu} = \beta \vert_{ T_{[2]} \times T_{[2]} } $,
$ \mu(f_T) = +1 $, $ \mu( \mathrm{rad}'(\beta) ) = \{ +1 \} $ and $ \mathrm{Arf}(\mu) = +1 $.

(1-h)
$ M_{n/2}(\mathbb{H}) \times M_{n/2}(\mathbb{H}) \cong M_2(\mathbb{R}) \otimes \ldots \otimes M_2(\mathbb{R}) \otimes \mathbb{H} \otimes M_2( \mathbb{R} \times \mathbb{R} ) $,
$ n = 2^m \geq 4 $.
The grading $\Gamma$ is determined up to isomorphism by $(T,\beta,\mu)$,
where $T$ is a subgroup of $G$ isomorphic to $ \mathbb{Z}_2^{2m-1} \times \mathbb{Z}_4 $,
$\beta$ is an alternating bicharacter on $T$ of type II,
and $\mu$ is a quadratic form on $T_{[2]}$ such that $ \beta_{\mu} = \beta \vert_{ T_{[2]} \times T_{[2]} } $,
$ \mu(f_T) = +1 $, $ \mu( \mathrm{rad}'(\beta) ) = \{ +1 \} $ and $ \mathrm{Arf}(\mu) = -1 $.

(1-i)
$ M_n(\mathbb{R}) \times M_{n/2}(\mathbb{H}) \cong M_2(\mathbb{R}) \otimes \ldots \otimes M_2(\mathbb{R}) \otimes [ M_2(\mathbb{R}) \times \mathbb{H} ] $,
$ n = 2^m \geq 2 $.
The grading $\Gamma$ is determined up to isomorphism by $(T,\beta,\mu)$,
where $T$ is a subgroup of $G$ isomorphic to $ \mathbb{Z}_2^{2m-1} \times \mathbb{Z}_4 $,
$\beta$ is an alternating bicharacter on $T$ of type II,
and $\mu$ is a quadratic form on $T_{[2]}$ such that $ \beta_{\mu} = \beta \vert_{ T_{[2]} \times T_{[2]} } $,
$ \mu(f_T) = +1 $ and $ \mu( \mathrm{rad}'(\beta) ) = \{ -1 \} $.

\begin{proof}
Denote by $I$ the unity of $\mathcal{D}$.
If $Z(\mathcal{D})$ has dimension $2$,
denote $ Z = i I $ in the case $ Z(\mathcal{D}) = \mathbb{C} $ (where $i$ is the imaginary unit)
and $ Z = (1,-1) I $ in the case $ Z(\mathcal{D}) = \mathbb{R} \times \mathbb{R} $.
Then $Z(\mathcal{D})$ is either $\mathbb{R}I$ or $\mathbb{R}I\oplus\mathbb{R}Z$, and it is easy to show that 
$Z$ is homogeneous and its degree $f$ is an element of order $2$
(see for example \cite[Lemma 14]{Rodrigo-2016}).

Now the classification is obtained by repeating
the arguments of \cite[Theorems 15 and 16]{Rodrigo-2016}.
Remark that $ f_{\beta} = f $
and that now the case $ \mu( f_{\beta} ) = +1 $ is possible.

As an example, let us recall the construction of
a graded division algebra $\mathcal{D}$
for a given datum $(T,\mu)$ in the case (1-c), that is,
$ T \cong \mathbb{Z}_2^{2m+1} $,
$\beta_{\mu}$ of type II and $ \mu(f_{\beta}) = -1 $.
We pick a basis $ \{ a_1 , b_1 , \ldots , a_m , b_m , f_{\beta} \} $
of $T$ as in Proposition \ref{prop:BasisAltBich}.
Changing if necessary $a_j$ or $b_j$
to $ a_j f_{\beta} $ or $ b_j f_{\beta} $, respectively,
we can get $ \mu(a_j) = \mu(b_j) = +1 $
($ j = 1 , \dots , m $).
Hence we can choose as the graded algebra
$ M_2(\mathbb{R}) \otimes \ldots \otimes M_2(\mathbb{R}) \otimes \mathbb{C} $,
as it is stated in the list,
where the support of the $j$-th factor is $ \langle a_j , b_j \rangle $
($ j = 1 , \dots , m $)
and the support of the last factor is $ \langle f_{\beta} \rangle $.

The only new argument is the fact that
$ M_2(\mathbb{R}) \otimes [ M_2(\mathbb{R}) \times \mathbb{H} ] $ and
$ \mathbb{H} \otimes [ M_2(\mathbb{R}) \times \mathbb{H} ] $
are in the same equivalence class, (1-i).
Indeed, consider the symplectic basis $ \{ a_1 , b_1 , a_2 , b_2 \} $
of $ \mathbb{Z}_2^2 \times ( \mathbb{Z}_2 \times \mathbb{Z}_4 ) $
relative to the grading of
$ \mathbb{H} \otimes [ M_2(\mathbb{R}) \times \mathbb{H} ] $,
that is, $a_1$ and $b_1$ generate the support of the grading on $\mathbb{H}$,
while $a_2$ and $b_2$ generate the support of the grading on $M_2(\mathbb{R}) \times \mathbb{H}$ 
playing the roles of $a$ and $b$ in Example \ref{exam:GradDim1}, so $b_2^2 \neq e $.
The quadratic form is determined by
$ \mu(a_1) = \mu(b_1) = -1 $
and $ \mu(a_2) = -1 $, $ \mu(b_2^2) = +1 $.
Then $ a_1' = a_1 a_2 $, $ b_1' = b_1 a_2 $,
$ a_2' = a_2 $, $ b_2' = a_1 b_1 b_2 $
form another symplectic basis, but now
$ \mu(a_1') = \mu(b_1') = +1 $ and still
$ \mu(a_2') = -1 $, $ \mu((b_2')^2) = +1 $.
Therefore, we can rewrite
$ \mathbb{H} \otimes [ M_2(\mathbb{R}) \times \mathbb{H} ] $ as
$ M_2(\mathbb{R}) \otimes [ M_2(\mathbb{R}) \times \mathbb{H} ] $
by renaming the elements of the group, so these graded algebras are equivalent.
\end{proof}

\begin{remark}
If $Z(\mathcal{D})$ is $\mathbb{R}\times\mathbb{R}$ then it must be nontrivially graded, so it is isomorphic
to the group algebra of a subgroup of $G$ of order $2$, namely, $\{e,f\}$ where $f=f_\beta$. 
Hence, $\mathcal{D}$ with its $G$-grading can be obtained from a simple algebra with a grading by 
the quotient group $G/\langle f\rangle$ by means of the \emph{loop construction} (see \cite{ABFP}):
(1-e) and (1-g) from (1-a); (1-f) and (1-h) from (1-b); and (1-i) from either (1-a) or (1-b).
\end{remark}

\section{Classification in the one-dimensional case}\label{sect:Dim1}

\begin{example}\label{exam:InvOneD}
We define involutions on $M_2(\mathbb{R})$,
$\mathbb{H}$, $\mathbb{C}$ and $M_2(\mathbb{C})$
that respect the gradings of Example \ref{exam:GradDim1}.
The notation with subscripts will make sense later on,
when the classification of this section is stated.
\begin{itemize}
\item
Let $\varphi_{\operatorname{(1-a-1)}}$ be
the matrix transpose on $M_2(\mathbb{R})$.
It is an orthogonal involution with signature $2$.
\item
Let $\varphi_{\operatorname{(1-a-2)}}$ be
the involution on $M_2(\mathbb{R})$ given by
$ \varphi_{\operatorname{(1-a-2)}} (X) = A^{-1} X^T A $,
where
\begin{equation}\label{eq:A}
A = \begin{pmatrix} 1 & 0 \\ 0 & -1 \end{pmatrix}.
\end{equation}
It is an orthogonal involution with signature $0$.
\item
Let $\varphi_{\operatorname{(1-a-3)}}$ be
the involution on $M_2(\mathbb{R})$
that acts as minus the identity on the matrices of trace zero,
and acts as the identity on the center of $M_2(\mathbb{R})$.
It is a symplectic involution.
\item
Let $\varphi_{\operatorname{(1-b-1)}}$ be
the standard conjugation on $\mathbb{H}$.
It is a symplectic involution with signature $1$.
\item
Let $\varphi_{\operatorname{(1-b-3)}}$ be
the involution on $\mathbb{H}$
that acts as the identity on $1$, $i$ and $j$,
and acts as minus the identity on $k$.
It is an orthogonal involution.
\item
Let $\varphi_{\operatorname{(1-c-1)}}$ be
the conjugation on $\mathbb{C}$.
It is an involution of the second kind and has signature $1$.
\item
Let $\varphi_{\operatorname{(1-c-3)}}$ be
the identity on $\mathbb{C}$.
It is an involution of the first kind and orthogonal.
\item
Let $\varphi_{\operatorname{(1-d-1)}}$ be
the matrix transpose on $M_2(\mathbb{C})$.
It is an involution of the first kind and orthogonal.
\item
Let $\varphi_{\operatorname{(1-d-3)}}$ be
the involution on $M_2(\mathbb{C})$ given by
$ \varphi_{\operatorname{(1-a-2)}} (X) = A^{-1} X^T A $,
with $A$ as in Equation \eqref{eq:A}.
It is an involution of the first kind and orthogonal.
\item
Let $\varphi_{\operatorname{(1-d-4)}}$ be
the involution on $M_2(\mathbb{C})$
that acts as minus the identity on the matrices of trace zero,
and acts as the identity on the center of $M_2(\mathbb{C})$.
It is an involution of the first kind and symplectic.
\end{itemize}
The involution $\varphi_{\operatorname{(1-a-1)}}$ will occur most frequently,
so we will abbreviate it as $\varphi_*$.
\end{example}

Let $G$ be an abelian group,
$\mathcal{D}$ a finite-dimensional simple real (associative) algebra,
and $\Gamma$ a division $G$-grading on $\mathcal{D}$ with homogeneous components of dimension $1$.
These gradings are classified in \cite[Theorems 15 and 16]{Rodrigo-2016};
there are four families of equivalence classes: (1-a), (1-b), (1-c) and (1-d).
We keep the same notation, so let $T$ be the support of $\Gamma$,
and let $ \beta : T \times T \rightarrow \{ \pm 1 \} $ be the alternating bicharacter given by the commutation relations.

Then, any antiautomorphism $\varphi$ of the $G$-graded algebra $\mathcal{D}$ is an involution.
We want to classify the pairs $(\Gamma,\varphi)$, up to isomorphism and up to equivalence.
The isomorphism classes are in bijective correspondence with the quadratic forms $\eta$ on $T$ such that $ \beta_{\eta} = \beta $.
Namely, the correspondence is given by the equation
\begin{equation}\label{eq:Eta}
\varphi(X_t) = \eta(t) X_t
\end{equation}
for all $ X_t \in \mathcal{D}_t $.
Now we give a list of the equivalence classes
together with a representative of every class,
and we also compile the classification up to isomorphism to serve as a reference:

\medskip

(1-a)
The grading $\Gamma$ on $ \mathcal{D} \cong M_n(\mathbb{R}) $ ($ n = 2^m \geq 1 $) was determined up to isomorphism by $(T,\mu)$,
where $T$ was a subgroup of $G$ isomorphic to $\mathbb{Z}_2^{2m}$,
and $\mu$ was a quadratic form on $T$ such that $\beta := \beta_{\mu}$ had type I and $ \mathrm{Arf}(\mu) = +1 $.
Now $(\Gamma,\varphi)$ is determined up to isomorphism by $(T,\mu,\eta)$,
where $\eta$ is a quadratic form on $T$ such that $ \beta_{\eta} = \beta $.
These isomorphism classes belong to one of the following three equivalence classes:
\begin{enumerate}
	\item[(1)] $ \eta = \mu $ ($ n = 2^m \geq 1 $).\\
	The involution $\varphi$ is orthogonal with signature $n$.\\
	A representative is
	$ \varphi_* \otimes \dots \otimes \varphi_* $
	(if $n=1$, $\varphi$ is just the identity on $ \mathbb{R} $).
	\item[(2)] $ \mathrm{Arf}(\eta) = +1 $ but $ \eta \neq \mu $ ($ n = 2^m \geq 2 $).\\
	The involution $\varphi$ is orthogonal with signature $0$.\\
	A representative is
	$ \varphi_* \otimes \dots \otimes \varphi_* \otimes \varphi_{\operatorname{(1-a-2)}} $.
	\item[(3)] $ \mathrm{Arf}(\eta) = -1 $ ($ n = 2^m \geq 2 $).\\
	The involution $\varphi$ is symplectic.\\
	A representative is
	$ \varphi_* \otimes \dots \otimes \varphi_* \otimes \varphi_{\operatorname{(1-a-3)}} $.
\end{enumerate}

\medskip

(1-b)
The grading $\Gamma$ on $ \mathcal{D} \cong M_{n/2}(\mathbb{H}) $ ($ n = 2^m \geq 2 $) was determined up to isomorphism by $(T,\mu)$,
where $T$ was a subgroup of $G$ isomorphic to $\mathbb{Z}_2^{2m}$,
and $\mu$ was a quadratic form on $T$ such that $\beta := \beta_{\mu}$ had type I and $ \mathrm{Arf}(\mu) = -1 $.
Now $(\Gamma,\varphi)$ is determined up to isomorphism by $(T,\mu,\eta)$,
where $\eta$ is a quadratic form on $T$ such that $ \beta_{\eta} = \beta $.
These isomorphism classes belong to one of the following three equivalence classes:
\begin{enumerate}
	\item[(1)] $ \eta = \mu $ ($ n = 2^m \geq 2 $).\\
	The involution $\varphi$ is symplectic with signature $n/2$.\\
	A representative is
	$ \varphi_* \otimes \dots \otimes \varphi_* \otimes \varphi_{\operatorname{(1-b-1)}} $.
	\item[(2)] $ \mathrm{Arf}(\eta) = -1 $ but $ \eta \neq \mu $ ($ n = 2^m \geq 4 $).\\
	The involution $\varphi$ is symplectic with signature $0$.\\
	A representative is
	$ \varphi_* \otimes \dots \otimes \varphi_* \otimes \varphi_{\operatorname{(1-a-2)}} \otimes \varphi_{\operatorname{(1-b-1)}} $.
	\item[(3)] $ \mathrm{Arf}(\eta) = +1 $ ($ n = 2^m \geq 2 $).\\
	The involution $\varphi$ is orthogonal.\\
	A representative is
	$ \varphi_* \otimes \dots \otimes \varphi_* \otimes \varphi_{\operatorname{(1-b-3)}} $.
\end{enumerate}

\medskip

(1-c)
The grading $\Gamma$ on $ \mathcal{D} \cong M_n(\mathbb{C}) $ ($ n = 2^m \geq 1 $) was determined up to isomorphism by $(T,\mu)$,
where $T$ was a subgroup of $G$ isomorphic to $\mathbb{Z}_2^{2m+1}$,
and $\mu$ was a quadratic form on $T$ such that $\beta := \beta_{\mu}$ had type II and $ \mu(f_{\beta}) = -1 $.
Now $(\Gamma,\varphi)$ is determined up to isomorphism by $(T,\mu,\eta)$,
where $\eta$ is a quadratic form on $T$ such that $ \beta_{\eta} = \beta $.
These isomorphism classes belong to one of the following four equivalence classes:
\begin{enumerate}
	\item[(1)] $ \eta = \mu $ ($ n = 2^m \geq 1 $).\\
	The involution $\varphi$ is of the second kind and has signature $n$.\\
	A representative is
	$ \varphi_* \otimes \dots \otimes \varphi_* \otimes \varphi_{\operatorname{(1-c-1)}} $.
	\item[(2)] $ \eta(f_{\beta}) = -1 $ but $ \eta \neq \mu $ ($ n = 2^m \geq 2 $).\\
	The involution $\varphi$ is of the second kind and has signature $0$.\\
	A representative is
	$ \varphi_* \otimes \dots \otimes \varphi_* \otimes \varphi_{\operatorname{(1-a-2)}} \otimes \varphi_{\operatorname{(1-c-1)}} $.
	\item[(3)] $ \eta(f_{\beta}) = +1 $ and $ \mathrm{Arf}(\eta) = +1 $ ($ n = 2^m \geq 1 $).\\
	The involution $\varphi$ is of the first kind and orthogonal.\\
	A representative is
	$ \varphi_* \otimes \dots \otimes \varphi_* \otimes \varphi_{\operatorname{(1-c-3)}} $.
	\item[(4)] $ \eta(f_{\beta}) = +1 $ and $ \mathrm{Arf}(\eta) = -1 $ ($ n = 2^m \geq 2 $).\\
	The involution $\varphi$ is of the first kind and symplectic.\\
	A representative is
	$ \varphi_* \otimes \dots \otimes \varphi_* \otimes \varphi_{\operatorname{(1-a-3)}} \otimes \varphi_{\operatorname{(1-c-3)}} $.
\end{enumerate}

\medskip

(1-d)
The grading $\Gamma$ on $ \mathcal{D} \cong M_n(\mathbb{C}) $ ($ n = 2^m \geq 2 $) was determined up to isomorphism by $(T,\beta,\mu)$,
where $T$ was a subgroup of $G$ isomorphic to $ \mathbb{Z}_2^{2m-1} \times \mathbb{Z}_4 $,
$\beta$ was an alternating bicharacter on $T$ of type II,
and $\mu$ was a quadratic form on $T_{[2]}$ such that $ \beta_{\mu} = \beta \vert_{ T_{[2]} \times T_{[2]} } $ and $ \mu(f_T) = -1 $.
Now $(\Gamma,\varphi)$ is determined up to isomorphism by $(T,\beta,\mu,\eta)$,
where $\eta$ is a quadratic form on $T$ such that $ \beta_{\eta} = \beta $ (so $ \eta(f_T) = +1 $ by Remark \ref{rem:fT}).
These isomorphism classes belong to one of the following four equivalence classes:
\begin{enumerate}
	\item[(1)] $ \eta( \mathrm{rad}'(\beta) ) = \{ +1 \} $ and $ \mathrm{Arf}(\eta) = +1 $ ($ n = 2^m \geq 2 $).\\
	The involution $\varphi$ is of the first kind and orthogonal.\\
	A representative is
	$ \varphi_* \otimes \dots \otimes \varphi_* \otimes \varphi_{\operatorname{(1-d-1)}} $.
	\item[(2)] $ \eta( \mathrm{rad}'(\beta) ) = \{ +1 \} $ and $ \mathrm{Arf}(\eta) = -1 $ ($ n = 2^m \geq 4 $).\\
	The involution $\varphi$ is of the first kind and symplectic.\\
	A representative is
	$ \varphi_* \otimes \dots \otimes \varphi_* \otimes \varphi_{\operatorname{(1-a-3)}} \otimes \varphi_{\operatorname{(1-d-1)}} $.
	\item[(3)] $ \eta( \mathrm{rad}'(\beta) ) = \{ -1 \} $ and $ \mathrm{Arf}(\eta) = +1 $ ($ n = 2^m \geq 2 $).\\
	The involution $\varphi$ is of the first kind and orthogonal.\\
	A representative is
	$ \varphi_* \otimes \dots \otimes \varphi_* \otimes \varphi_{\operatorname{(1-d-3)}} $.
	\item[(4)] $ \eta( \mathrm{rad}'(\beta) ) = \{ -1 \} $ and $ \mathrm{Arf}(\eta) = -1 $ ($ n = 2^m \geq 2 $).\\
	The involution $\varphi$ is of the first kind and symplectic.\\
	A representative is
	$ \varphi_* \otimes \dots \otimes \varphi_* \otimes \varphi_{\operatorname{(1-d-4)}} $.
\end{enumerate}

\begin{proof}
Define $ \eta : T \rightarrow \mathbb{R}^{\times} $
by Equation \eqref{eq:Eta}.
For all $ t \in T $,
we can pick $X_t$ so that $ X_t^8 = +I $,
hence $ \eta(t)^8 = +1 $;
therefore $\eta$ takes values in $ \{ \pm 1 \} $
and $\varphi$ is an involution.
The fact that $\varphi$ reverses the order of the product
is equivalent to $\eta$ being a quadratic form
with $ \beta_{\eta} = \beta $; indeed:
\[
\eta(uv) X_v X_u
= \varphi( X_v X_u )
= \varphi(X_u) \varphi(X_v)
= \eta(u) \eta(v) \beta(u,v) X_v X_u .
\]
Thus we have a bijective correspondence between
the isomorphism classes of pairs $(\Gamma,\varphi)$ and
the quadratic forms $\eta$ on $T$ such that $ \beta_{\eta} = \beta $.

Involutions belonging to (1-a-1) or (1-a-2)
are not equivalent to those in (1-a-3),
because of the Arf invariant.
The involution (1-a-1) is determined by Equation \eqref{eq:QuadrForm},
so it is not equivalent to the involutions that belong to (1-a-2),
in other words, it is a distinguished involution of the grading.
Considering also $\eta(f_{\beta})$
and $ \eta( \mathrm{rad}'(\beta) ) $,
we see that the rest of the equivalence classes of the list do not overlap.

We know that there exist quadratic forms $\eta$
for the indicated values of $n$
because of Section \ref{sect:GradQF}.

The tricky point is to prove that
involutions that lie in the same item of the list are equivalent.
The idea is to write any $\varphi$ in a given equivalence class
as the representative that we indicated in the list.
Let us start with the case (1-a-2),
so assume that $ T \cong \mathbb{Z}_2^{2m} $,
$\beta$ has type I,
$ \mathrm{Arf}(\mu) = \mathrm{Arf}(\eta) = +1 $,
and $ \mu \neq \eta $.
By Lemmas \ref{lem:index2} and \ref{lem:Ind2SubgrTyp1},
there exists $ b_1 \in T $ such that
$ b_1^{\perp} = \{ t \in T \mid \mu(t) = \eta(t) \} $.
We are going to prove that $ \mu(b_1) = +1 $.
Take $ c \in T \setminus b_1^{\perp} $,
so $ T \cong \langle c \rangle \times b_1^{\perp} $.
Because of Equation \eqref{eq:BetaMu},
$\mu$ takes the value $-1$ either once or three times
on $ \langle c,b_1 \rangle $.
Since $\eta$ has the same Arf invariant as $\mu$ and
$ T \cong \langle c,b_1 \rangle \times \langle c,b_1 \rangle^{\perp} $,
$\eta$ takes the value $-1$ on $ \langle c,b_1 \rangle $
as many times as $\mu$.
This number cannot be three,
because $ \mu(c) \neq \eta(c) $,
so $ \mu(b_1) = \eta(b_1) = +1 $.

We can take $ a_1 \in T $, and then inductively
$ a_2 , b_2 , \ldots , a_m , b_m \in T $ so that
$ \{ a_1 , \allowbreak b_1 , \allowbreak \ldots , \allowbreak a_m , \allowbreak b_m \} $
is a symplectic basis as defined before Proposition \ref{prop:BasisAltBich}
(follow, for example, the arguments in \cite[Equation (2.6) on p.~36]{EK-2013}).
Moreover, since $ \mathrm{Arf}(\mu) = +1 $ and $ \mu(b_1) = +1 $,
we can argue as in the last paragraph of
\cite[proof of Theorems 15 and 16]{Rodrigo-2016}
and assume that our symplectic basis satisfies
\[
\mu(a_1) = \mu(b_1) = \ldots = \mu(a_m) = \mu(b_m) = +1.
\]
By construction, this implies
\[
\eta(a_1) = -1
\quad\text{and}\quad
\eta(b_1) = \ldots = \eta(a_m) = \eta(b_m) = +1.
\]
We have shown that any $\varphi$ in (1-a-2) can be
written as $ \varphi_{\operatorname{(1-a-2)}} \otimes \varphi_* \otimes \dots \otimes \varphi_* $,
thus they are all equivalent.

The same reasoning works for (1-a-3), but note
that now $ \mu(b_1) = \eta(b_1) = -1 $.
Analogously for (1-b-2), (1-b-3) and (1-c-2), but
in this last case we use Lemma \ref{lem:Ind2SubgrTyp2}
instead of Lemma \ref{lem:Ind2SubgrTyp1},
and we may replace $b_1$ by $ b_1 f_{\beta} $ so that $ \mu(b_1) = +1 $.
In the cases (1-c-3), (1-c-4), (1-d-1), (1-d-2), (1-d-3) and (1-d-4),
we cannot apply Lemma \ref{lem:Ind2SubgrTyp2},
but in fact they are easier, because
$ \mu(f_{\beta}) = -1 $ whereas $ \eta(f_{\beta}) = +1 $.
We can first pick $ a_1 , b_1 , \ldots , a_m , b_m \in T $
so that $\eta$ takes the values that we want on them.
Then, changing, if necessary, the $a_i$ and $b_j$
to $ a_i f_{\beta} $ and $ b_j f_{\beta} $,
we can also select the values taken by $\mu$.
For example, in the case (1-c-3),
$\eta$ is a quadratic form on $ T \cong \mathbb{Z}_2^{2m+1} $
such that $\beta_{\eta}$ has type II,
$ \eta(f_{\beta}) = +1 $ and $ \mathrm{Arf}(\eta) = +1 $.
This means that $\eta$ belongs to the item (1-e)
of the list of Section \ref{sect:GradQF},
hence there exists a basis
$ \{ a_1 , b_1 , \ldots , a_m , b_m , f_{\beta} \} $ of $T$
as in Proposition \ref{prop:BasisAltBich} such that
$ \eta(a_1) = \eta(b_1) = \dots = \eta(a_m) = \eta(b_m) = 1 $
(and $ \eta(f_{\beta}) = 1 $).
We can assume without loss of generality that also
$ \mu(a_1) = \mu(b_1) = \dots = \mu(a_m) = \mu(b_m) = 1 $
(and $ \mu(f_{\beta}) = -1 $).
Therefore any $\varphi$ can be written as
$ \varphi_* \otimes \dots \otimes \varphi_* \otimes \varphi_{\operatorname{(1-c-3)}} $.

Finally, in order to compute the signatures
it is enough to apply Lemma \ref{lem:SignReal}
to the representative of every equivalence class,
since we already calculated the signature
of each factor in Example \ref{exam:InvOneD}.
\end{proof}

\begin{remark}\label{rem:all_inv_1}
Recall from Lemma \ref{lem:InnAut} that,
given an involution $\varphi$
on the graded algebra $\mathcal{D}$,
we can obtain the rest of the involutions
(of the same kind, if $ \mathcal{D} \cong M_n(\mathbb{C}) $)
as $ \mathrm{Int}(X_u) \circ \varphi $,
where $u$ runs through $T$.
If $\eta$ is the quadratic form on $T$
corresponding to $\varphi$,
then the quadratic form $ \eta_u : T \rightarrow \{ \pm 1 \} $
corresponding to $ \mathrm{Int}(X_u) \circ \varphi $
is given by
$ \eta_u(v) = \beta(u,v) \eta(v) = \eta(uv) \eta(u) $.
In particular,
$ \mathrm{Arf}(\eta_u) = \mathrm{Arf}(\eta) \eta(u) $
if the Arf invariant is defined.
\end{remark}

\section{Classification in the two-dimensional non-complex case}\label{sect:Dim2NonComplex}

\begin{example}
Consider the division grading by the group $\mathbb{Z}_2$ on the algebra $M_2(\mathbb{R})$
obtained by coarsening of the grading of Example \ref{exam:GradDim1}:
\[ M_2(\mathbb{R}) = \left[
\mathbb{R} \begin{pmatrix} 1 & 0 \\ 0 & 1 \end{pmatrix} \oplus
\mathbb{R} \begin{pmatrix} 0 & -1 \\ 1 & 0 \end{pmatrix} \right] \oplus \left[
\mathbb{R} \begin{pmatrix} 0 & 1 \\ 1 & 0 \end{pmatrix} \oplus
\mathbb{R} \begin{pmatrix} -1 & 0 \\ 0 & 1 \end{pmatrix} \right] . \]
When $M_2(\mathbb{R})$ is endowed with this grading,
we denote the involutions of Example \ref{exam:InvOneD} as:
\begin{itemize}
\item
$ \varphi_{\operatorname{(2-a-1)}} = \varphi_{\operatorname{(1-a-2)}} $;
\item
$ \varphi_{\operatorname{(2-a-3)}} = \varphi_{\operatorname{(1-a-1)}} $;
\item
$ \varphi_{\operatorname{(2-a-5)}} = \varphi_{\operatorname{(1-a-3)}} $.
\end{itemize}
Analogously, if $\mathbb{H}$ is $\mathbb{Z}_2$-graded as
$ \mathbb{H} = \left[ \mathbb{R}1 \oplus \mathbb{R}i \right]
\oplus \left[ \mathbb{R}j \oplus \mathbb{R}k \right] $, we denote:
\begin{itemize}
\item
$ \varphi_{\operatorname{(2-b-2)}} = \varphi_{\operatorname{(1-b-3)}} $;
\item
$ \varphi_{\operatorname{(2-b-3)}} = \varphi_{\operatorname{(1-b-1)}} $;
\item
we also denote $\varphi_{\operatorname{(2-b-5)}}$
the involution on $\mathbb{H}$
that acts as the identity on $1$, $j$ and $k$,
and acts as minus the identity on $i$;
it is an orthogonal involution.
\end{itemize}
Finally, consider the division grading by the group $\mathbb{Z}_4$ on the algebra $M_2(\mathbb{C})$
obtained by the coarsening of the grading of Figure \ref{fig:DivGrad} that joins,
for all $ t \in \langle a \rangle \times \langle b \rangle
\cong \mathbb{Z}_2 \times \mathbb{Z}_4 $,
the homogeneous component of degree $t$
to the homogeneous component of degree $ a b^2 t $.
When $M_2(\mathbb{C})$ is endowed with this grading,
we denote the involutions of Example \ref{exam:InvOneD} as:
\begin{itemize}
\item
$ \varphi_{\operatorname{(2-e-1)}} = \varphi_{\operatorname{(1-d-1)}} $;
\item
$ \varphi_{\operatorname{(2-e-3)}} = \varphi_{\operatorname{(1-d-3)}} $;
\item
$ \varphi_{\operatorname{(2-e-4)}} = \varphi_{\operatorname{(1-d-4)}} $.
\end{itemize}
\end{example}

Let $G$ be an abelian group,
$\mathcal{D}$ a finite-dimensional simple real (associative) algebra,
and $\Gamma$ a division $G$-grading on $\mathcal{D}$ with homogeneous components of dimension $2$
such that the identity component does not coincide with the center of $\mathcal{D}$.
These gradings are classified in \cite[Theorems 22 and 23]{Rodrigo-2016};
there are five families of equivalence classes: (2-a), (2-b), (2-c), (2-d) and (2-e).
We keep the same notation,
so write $ \mathcal{D}_e = \mathbb{R}I \oplus \mathbb{R}J $ ($ \cong \mathbb{C} $), where $I$ is the unity of $\mathcal{D}$ and $J^2=-I$;
and let $T$ be the support of $\Gamma$, $K$ the support of the centralizer of the identity component,
and $ \beta : K \times K \rightarrow \{ \pm 1 \} $ the alternating bicharacter given by the commutation relations in the centralizer of the identity component.

Then,
for any antiautomorphism $\varphi$ of the $G$-graded algebra $\mathcal{D}$, either $\varphi(J)=+J$ or $\varphi(J)=-J$.
We want to classify the pairs $(\Gamma,\varphi)$, up to isomorphism and up to equivalence, when $\varphi$ is an involution.
In the case $\varphi(J)=+J$, any antiautomorphism is an involution,
and there is exactly one proper refinement of $\Gamma$ compatible with a given involution;
the isomorphism classes are in bijective correspondence with the quadratic forms $\eta$ on $K$ such that $ \beta_{\eta} = \beta $
(and, in the case (2-e), $ \eta(f_T) = +1 $)
by means of the equation:
\begin{equation}
\varphi(X_k) = \eta(k) X_k
\end{equation}
for all $ k \in K $ and $ X_k \in \mathcal{D}_k $.
In the case $\varphi(J)=-J$, there are antiautomorphisms that are not involutions,
but any proper refinement of $\Gamma$ is compatible with a given involution;
the isomorphism classes
of involutions
are in bijective correspondence with the nice maps $\omega$ on $ T \setminus K $ such that $ \beta_{\omega} = \beta $
(and, in the case (2-e), $ \omega(f_T) = +1 $)
by means of the equation:
\begin{equation}
\varphi(X_t) = \omega(t) X_t
\end{equation}
for all $ t \in T \setminus K $ and $ X_t \in \mathcal{D}_t $.
Now we give a list of the equivalence classes
together with a representative of every class,
and we also compile the classification up to isomorphism to serve as a reference:

\medskip

(2-a)
The grading $\Gamma$ on $ \mathcal{D} \cong M_n(\mathbb{R}) $ ($ n = 2^m \geq 2 $) was determined up to isomorphism by $(T,K,\nu)$,
where $T$ was a subgroup of $G$ isomorphic to $\mathbb{Z}_2^{2m-1}$,
$K$ was a subgroup of $T$ of index $2$,
and $\nu$ was a nice map on $ T \setminus K $ such that $\beta := \beta_{\nu}$ had type I and $ \mathrm{Arf}(\nu) = +1 $.
Now, in the case $\varphi(J)=+J$, $(\Gamma,\varphi)$ is determined up to isomorphism by $(T,K,\nu,\eta)$,
where $\eta$ is a quadratic form on $K$ such that $ \beta_{\eta} = \beta $.
These isomorphism classes belong to one of the following two equivalence classes:
\begin{enumerate}
	\item[(1)] $ \mathrm{Arf}(\eta) = +1 $ ($ n = 2^m \geq 2 $).\\
	The involution $\varphi$ is orthogonal with signature $0$.\\
	A representative is
	$ \varphi_{\operatorname{(2-a-1)}} \otimes \varphi_* \otimes \dots \otimes \varphi_* $.
	\item[(2)] $ \mathrm{Arf}(\eta) = -1 $ ($ n = 2^m \geq 4 $).\\
	The involution $\varphi$ is symplectic.\\
	A representative is
	$ \varphi_{\operatorname{(2-a-1)}} \otimes \varphi_* \otimes \dots \otimes \varphi_* \otimes \varphi_{\operatorname{(1-a-3)}} $.
\end{enumerate}
On the other hand, in the case $\varphi(J)=-J$, $(\Gamma,\varphi)$ is determined up to isomorphism by
$ ( T , \allowbreak K , \allowbreak \nu , \allowbreak \omega ) $,
where $\omega$ is a nice map on $ T \setminus K $ such that $ \beta_{\omega} = \beta $.
These isomorphism classes belong to one of the following four equivalence classes:
\begin{enumerate}
	\item[(3)] $ \omega = \nu $ ($ n = 2^m \geq 2 $).\\
	The involution $\varphi$ is orthogonal with signature $n$.\\
	A representative is
	$ \varphi_{\operatorname{(2-a-3)}} \otimes \varphi_* \otimes \dots \otimes \varphi_* $.
	\item[(4)] $ \mathrm{Arf}(\omega) = +1 $ but $ \omega \neq \nu $ ($ n = 2^m \geq 4 $).\\
	The involution $\varphi$ is orthogonal with signature $0$.\\
	A representative is
	$ \varphi_{\operatorname{(2-a-3)}} \otimes \varphi_* \otimes \dots \otimes \varphi_* \otimes \varphi_{\operatorname{(1-a-2)}} $.
	\item[(5)] $ \omega = -\nu $ ($ n = 2^m \geq 2 $).\\
	The involution $\varphi$ is symplectic.\\
	A representative is
	$ \varphi_{\operatorname{(2-a-5)}} \otimes \varphi_* \otimes \dots \otimes \varphi_* $.
	\item[(6)] $ \mathrm{Arf}(\omega) = -1 $ but $ \omega \neq -\nu $ ($ n = 2^m \geq 4 $).\\
	The involution $\varphi$ is symplectic.\\
	A representative is
	$ \varphi_{\operatorname{(2-a-5)}} \otimes \varphi_* \otimes \dots \otimes \varphi_* \otimes \varphi_{\operatorname{(1-a-2)}} $.
\end{enumerate}

\medskip

(2-b)
The grading $\Gamma$ on $ \mathcal{D} \cong M_{n/2}(\mathbb{H}) $ ($ n = 2^m \geq 2 $) was determined up to isomorphism by $(T,K,\nu)$,
where $T$ was a subgroup of $G$ isomorphic to $\mathbb{Z}_2^{2m-1}$,
$K$ was a subgroup of $T$ of index $2$,
and $\nu$ was a nice map on $ T \setminus K $ such that $\beta := \beta_{\nu}$ had type I and $ \mathrm{Arf}(\nu) = -1 $.
Now, in the case $\varphi(J)=+J$, $(\Gamma,\varphi)$ is determined up to isomorphism by $(T,K,\nu,\eta)$,
where $\eta$ is a quadratic form on $K$ such that $ \beta_{\eta} = \beta $.
These isomorphism classes belong to one of the following two equivalence classes:
\begin{enumerate}
	\item[(1)] $ \mathrm{Arf}(\eta) = -1 $ ($ n = 2^m \geq 4 $).\\
	The involution $\varphi$ is symplectic with signature $0$.\\
	A representative is
	$ \varphi_{\operatorname{(2-b-2)}} \otimes \varphi_* \otimes \dots \otimes \varphi_* \otimes \varphi_{\operatorname{(1-a-3)}} $.
	\item[(2)] $ \mathrm{Arf}(\eta) = +1 $ ($ n = 2^m \geq 2 $).\\
	The involution $\varphi$ is orthogonal.\\
	A representative is
	$ \varphi_{\operatorname{(2-b-2)}} \otimes \varphi_* \otimes \dots \otimes \varphi_* $.
\end{enumerate}
On the other hand, in the case $\varphi(J)=-J$, $(\Gamma,\varphi)$ is determined up to isomorphism by
$ ( T , \allowbreak K , \allowbreak \nu , \allowbreak \omega ) $,
where $\omega$ is a nice map on $ T \setminus K $ such that $ \beta_{\omega} = \beta $.
These isomorphism classes belong to one of the following four equivalence classes:
\begin{enumerate}
	\item[(3)] $ \omega = \nu $ ($ n = 2^m \geq 2 $).\\
	The involution $\varphi$ is symplectic with signature $n/2$.\\
	A representative is
	$ \varphi_{\operatorname{(2-b-3)}} \otimes \varphi_* \otimes \dots \otimes \varphi_* $.
	\item[(4)] $ \mathrm{Arf}(\omega) = -1 $ but $ \omega \neq \nu $ ($ n = 2^m \geq 4 $).\\
	The involution $\varphi$ is symplectic with signature $0$.\\
	A representative is
	$ \varphi_{\operatorname{(2-b-3)}} \otimes \varphi_* \otimes \dots \otimes \varphi_* \otimes \varphi_{\operatorname{(1-a-2)}} $.
	\item[(5)] $ \omega = -\nu $ ($ n = 2^m \geq 2 $).\\
	The involution $\varphi$ is orthogonal.\\
	A representative is
	$ \varphi_{\operatorname{(2-b-5)}} \otimes \varphi_* \otimes \dots \otimes \varphi_* $.
	\item[(6)] $ \mathrm{Arf}(\omega) = +1 $ but $ \omega \neq -\nu $ ($ n = 2^m \geq 4 $).\\
	The involution $\varphi$ is orthogonal.\\
	A representative is
	$ \varphi_{\operatorname{(2-b-5)}} \otimes \varphi_* \otimes \dots \otimes \varphi_* \otimes \varphi_{\operatorname{(1-a-2)}} $.
\end{enumerate}

\medskip

(2-c)
The grading $\Gamma$ on $ \mathcal{D} \cong M_n(\mathbb{C}) $ ($ n = 2^m \geq 2 $) was determined up to isomorphism by $(T,K,\nu)$,
where $T$ was a subgroup of $G$ isomorphic to $\mathbb{Z}_2^{2m}$,
$K$ was a subgroup of $T$ of index $2$,
and $\nu$ was a nice map on $ T \setminus K $ such that $\beta := \beta_{\nu}$ had type II and $ \nu(f_{\beta}) = -1 $.
Now, in the case $\varphi(J)=+J$, $(\Gamma,\varphi)$ is determined up to isomorphism by $(T,K,\nu,\eta)$,
where $\eta$ is a quadratic form on $K$ such that $ \beta_{\eta} = \beta $.
These isomorphism classes belong to one of the following three equivalence classes:
\begin{enumerate}
	\item[(1)] $ \eta(f_{\beta}) = -1 $ ($ n = 2^m \geq 2 $).\\
	The involution $\varphi$ is of the second kind and has signature $0$.\\
	A representative is
	$ \varphi_{\operatorname{(2-a-1)}} \otimes \varphi_* \otimes \dots \otimes \varphi_* \otimes \varphi_{\operatorname{(1-c-1)}} $.
	\item[(2)] $ \eta(f_{\beta}) = +1 $ and $ \mathrm{Arf}(\eta) = +1 $ ($ n = 2^m \geq 2 $).\\
	The involution $\varphi$ is of the first kind and orthogonal.\\
	A representative is
	$ \varphi_{\operatorname{(2-a-1)}} \otimes \varphi_* \otimes \dots \otimes \varphi_* \otimes \varphi_{\operatorname{(1-c-3)}} $.
	\item[(3)] $ \eta(f_{\beta}) = +1 $ and $ \mathrm{Arf}(\eta) = -1 $ ($ n = 2^m \geq 4 $).\\
	The involution $\varphi$ is of the first kind and symplectic.\\
	A representative is
	$ \varphi_{\operatorname{(2-a-1)}} \otimes \varphi_* \otimes \dots \otimes \varphi_*
	\otimes \varphi_{\operatorname{(1-a-3)}} \otimes \varphi_{\operatorname{(1-c-3)}} $.
\end{enumerate}
On the other hand, in the case $\varphi(J)=-J$, $(\Gamma,\varphi)$ is determined up to isomorphism by
$ ( T , \allowbreak K , \allowbreak \nu , \allowbreak \omega ) $,
where $\omega$ is a nice map on $ T \setminus K $ such that $ \beta_{\omega} = \beta $.
These isomorphism classes belong to one of the following five equivalence classes:
\begin{enumerate}
	\item[(4)] $ \omega = \nu $ ($ n = 2^m \geq 2 $).\\
	The involution $\varphi$ is of the second kind and has signature $n$.\\
	A representative is
	$ \varphi_{\operatorname{(2-a-3)}} \otimes \varphi_* \otimes \dots \otimes \varphi_* \otimes \varphi_{\operatorname{(1-c-1)}} $.
	\item[(5)] $ \omega = -\nu $ ($ n = 2^m \geq 2 $).\\
	The involution $\varphi$ is of the second kind and has signature $0$.\\
	A representative is
	$ \varphi_{\operatorname{(2-a-5)}} \otimes \varphi_* \otimes \dots \otimes \varphi_* \otimes \varphi_{\operatorname{(1-c-1)}} $.
	\item[(6)] $ \omega(f_{\beta}) = -1 $ but $ \omega \neq \pm \nu $ ($ n = 2^m \geq 4 $).\\
	The involution $\varphi$ is of the second kind and has signature $0$.\\
	A representative is
	$ \varphi_{\operatorname{(2-a-3)}} \otimes \varphi_* \otimes \dots \otimes \varphi_*
	\otimes \varphi_{\operatorname{(1-a-2)}} \otimes \varphi_{\operatorname{(1-c-1)}} $.
	\item[(7)] $ \omega(f_{\beta}) = +1 $ and $ \mathrm{Arf}(\omega) = +1 $ ($ n = 2^m \geq 2 $).\\
	The involution $\varphi$ is of the first kind and orthogonal.\\
	A representative is
	$ \varphi_{\operatorname{(2-a-3)}} \otimes \varphi_* \otimes \dots \otimes \varphi_* \otimes \varphi_{\operatorname{(1-c-3)}} $.
	\item[(8)] $ \omega(f_{\beta}) = +1 $ and $ \mathrm{Arf}(\omega) = -1 $ ($ n = 2^m \geq 2 $).\\
	The involution $\varphi$ is of the first kind and symplectic.\\
	A representative is
	$ \varphi_{\operatorname{(2-a-5)}} \otimes \varphi_* \otimes \dots \otimes \varphi_* \otimes \varphi_{\operatorname{(1-c-3)}} $.
\end{enumerate}

\medskip

(2-d)
The grading $\Gamma$ on $ \mathcal{D} \cong M_n(\mathbb{C}) $ ($ n = 2^m \geq 4 $) was determined up to isomorphism by $(T,K,\beta,\nu)$,
where $T$ was a subgroup of $G$ isomorphic to $ \mathbb{Z}_2^{2m-2} \times \mathbb{Z}_4 $,
$K$ was a subgroup of $T$ of index $2$ but different from $T_{[2]}$,
$\beta$ was an alternating bicharacter on $K$ of type II,
and $\nu$ was a nice map on $ T_{[2]} \setminus K_{[2]} $ such that $ \beta_{\nu} = \beta \vert_{ K_{[2]} \times K_{[2]} } $ and $ \nu(f_T) = -1 $.
Now, in the case $\varphi(J)=+J$, $(\Gamma,\varphi)$ is determined up to isomorphism by $(T,K,\beta,\nu,\eta)$,
where $\eta$ is a quadratic form on $K$ such that $ \beta_{\eta} = \beta $ (so $ \eta(f_T) = +1 $).
These isomorphism classes belong to one of the following four equivalence classes:
\begin{enumerate}
	\item[(1)] $ \eta( \mathrm{rad}'(\beta) ) = \{ +1 \} $ and $ \mathrm{Arf}(\eta) = +1 $ ($ n = 2^m \geq 4 $).\\
	The involution $\varphi$ is of the first kind and orthogonal.\\
	A representative is
	$ \varphi_{\operatorname{(2-a-1)}} \otimes \varphi_* \otimes \dots \otimes \varphi_* \otimes \varphi_{\operatorname{(1-d-1)}} $.
	\item[(2)] $ \eta( \mathrm{rad}'(\beta) ) = \{ +1 \} $ and $ \mathrm{Arf}(\eta) = -1 $ ($ n = 2^m \geq 8 $).\\
	The involution $\varphi$ is of the first kind and symplectic.\\
	A representative is
	$ \varphi_{\operatorname{(2-a-1)}} \otimes \varphi_* \otimes \dots \otimes \varphi_*
	\otimes \varphi_{\operatorname{(1-a-3)}} \otimes \varphi_{\operatorname{(1-d-1)}} $.
	\item[(3)] $ \eta( \mathrm{rad}'(\beta) ) = \{ -1 \} $ and $ \mathrm{Arf}(\eta) = +1 $ ($ n = 2^m \geq 4 $).\\
	The involution $\varphi$ is of the first kind and orthogonal.\\
	A representative is
	$ \varphi_{\operatorname{(2-a-1)}} \otimes \varphi_* \otimes \dots \otimes \varphi_* \otimes \varphi_{\operatorname{(1-d-3)}} $.
	\item[(4)] $ \eta( \mathrm{rad}'(\beta) ) = \{ -1 \} $ and $ \mathrm{Arf}(\eta) = -1 $ ($ n = 2^m \geq 4 $).\\
	The involution $\varphi$ is of the first kind and symplectic.\\
	A representative is
	$ \varphi_{\operatorname{(2-a-1)}} \otimes \varphi_* \otimes \dots \otimes \varphi_* \otimes \varphi_{\operatorname{(1-d-4)}} $.
\end{enumerate}
On the other hand, in the case $\varphi(J)=-J$, $(\Gamma,\varphi)$ is determined up to isomorphism by
$ ( T , \allowbreak K , \allowbreak \beta , \allowbreak \nu , \allowbreak \omega ) $,
where $\omega$ is a nice map on $ T \setminus K $ such that $ \beta_{\omega} = \beta $ (so $ \omega(f_T) = +1 $).
These isomorphism classes belong to one of the following four equivalence classes:
\begin{enumerate}
	\item[(5)] $ \omega( \mathrm{rad}'(\beta) ) = \{ +1 \} $ and $ \mathrm{Arf}(\omega) = +1 $ ($ n = 2^m \geq 4 $).\\
	The involution $\varphi$ is of the first kind and orthogonal.\\
	A representative is
	$ \varphi_{\operatorname{(2-a-3)}} \otimes \varphi_* \otimes \dots \otimes \varphi_* \otimes \varphi_{\operatorname{(1-d-1)}} $.
	\item[(6)] $ \omega( \mathrm{rad}'(\beta) ) = \{ +1 \} $ and $ \mathrm{Arf}(\omega) = -1 $ ($ n = 2^m \geq 8 $).\\
	The involution $\varphi$ is of the first kind and symplectic.\\
	A representative is
	$ \varphi_{\operatorname{(2-a-3)}} \otimes \varphi_* \otimes \dots \otimes \varphi_*
	\otimes \varphi_{\operatorname{(1-a-3)}} \otimes \varphi_{\operatorname{(1-d-1)}} $.
	\item[(7)] $ \omega( \mathrm{rad}'(\beta) ) = \{ -1 \} $ and $ \mathrm{Arf}(\omega) = +1 $ ($ n = 2^m \geq 4 $).\\
	The involution $\varphi$ is of the first kind and orthogonal.\\
	A representative is
	$ \varphi_{\operatorname{(2-a-3)}} \otimes \varphi_* \otimes \dots \otimes \varphi_* \otimes \varphi_{\operatorname{(1-d-3)}} $.
	\item[(8)] $ \omega( \mathrm{rad}'(\beta) ) = \{ -1 \} $ and $ \mathrm{Arf}(\omega) = -1 $ ($ n = 2^m \geq 4 $).\\
	The involution $\varphi$ is of the first kind and symplectic.\\
	A representative is
	$ \varphi_{\operatorname{(2-a-3)}} \otimes \varphi_* \otimes \dots \otimes \varphi_* \otimes \varphi_{\operatorname{(1-d-4)}} $.
\end{enumerate}

\medskip

(2-e)
The grading $\Gamma$ on $ \mathcal{D} \cong M_n(\mathbb{C}) $ ($ n = 2^m \geq 2 $) was determined up to isomorphism by $(T,[\nu])$,
where $T$ was a subgroup of $G$ isomorphic to $ \mathbb{Z}_2^{2m-2} \times \mathbb{Z}_4 $ ($ K = T_{[2]} $),
and $[\nu]$ was an equivalence class of nice maps $\nu$ on $ T \setminus T_{[2]} $ such that
$\beta := \beta_{\nu}$ had type II, $ f_{\beta} = f_T $ and $ \nu(f_T) = -1 $,
with the equivalence relation $ \nu \sim \nu' $ if either $ \nu' = \nu $ or $ \nu' = -\nu $.
Now, in the case $\varphi(J)=+J$, $(\Gamma,\varphi)$ is determined up to isomorphism by $(T,[\nu],\eta)$,
where $\eta$ is a quadratic form on $T_{[2]}$ such that $ \beta_{\eta} = \beta $ and $ \eta(f_T) = +1 $.
These isomorphism classes belong to one of the following two equivalence classes:
\begin{enumerate}
	\item[(1)] $ \mathrm{Arf}(\eta) = +1 $ ($ n = 2^m \geq 2 $).\\
	The involution $\varphi$ is of the first kind and orthogonal.\\
	A representative is
	$ \varphi_{\operatorname{(2-e-1)}} \otimes \varphi_* \otimes \dots \otimes \varphi_* $.
	\item[(2)] $ \mathrm{Arf}(\eta) = -1 $ ($ n = 2^m \geq 4 $).\\
	The involution $\varphi$ is of the first kind and symplectic.\\
	A representative is
	$ \varphi_{\operatorname{(2-e-1)}} \otimes \varphi_* \otimes \dots \otimes \varphi_* \otimes \varphi_{\operatorname{(1-a-3)}} $.
\end{enumerate}
On the other hand, in the case $\varphi(J)=-J$, $(\Gamma,\varphi)$ is determined up to isomorphism by $(T,[\nu],\omega)$,
where $\omega$ is a nice map on $ T \setminus T_{[2]} $ such that $ \beta_{\omega} = \beta $ and $ \omega(f_T) = +1 $.
These isomorphism classes belong to one of the following two equivalence classes:
\begin{enumerate}
	\item[(3)] $ \mathrm{Arf}(\omega) = +1 $ ($ n = 2^m \geq 2 $).\\
	The involution $\varphi$ is of the first kind and orthogonal.\\
	A representative is
	$ \varphi_{\operatorname{(2-e-3)}} \otimes \varphi_* \otimes \dots \otimes \varphi_* $.
	\item[(4)] $ \mathrm{Arf}(\omega) = -1 $ ($ n = 2^m \geq 2 $).\\
	The involution $\varphi$ is of the first kind and symplectic.\\
	A representative is
	$ \varphi_{\operatorname{(2-e-4)}} \otimes \varphi_* \otimes \dots \otimes \varphi_* $.
\end{enumerate}

\begin{proof}
Let $ g \in T \setminus K$.
We start with the case (2-a).
We know from \cite[Theorem 22]{Rodrigo-2016} that
we can write $\mathcal{D}$ as follows:
\begin{equation}
( \mathcal{D}_e \oplus \mathcal{D}_g )
\otimes_{\mathbb{R}}
C_{\mathcal{D}} ( \mathcal{D}_e \oplus \mathcal{D}_g ).
\end{equation}
Recall that
$ C_{\mathcal{D}} ( \mathcal{D}_e \oplus \mathcal{D}_g ) $
is a subalgebra isomorphic to
$M_{n/2}(\mathbb{R})$ or $M_{n/4}(\mathbb{H})$,
endowed with a division grading
whose homogeneous components have dimension $1$.
Since $ \varphi ( \mathcal{D}_e \oplus \mathcal{D}_g ) =
\mathcal{D}_e \oplus \mathcal{D}_g $, also
$ \varphi ( C_{\mathcal{D}} ( \mathcal{D}_e \oplus \mathcal{D}_g ) )
= C_{\mathcal{D}} ( \mathcal{D}_e \oplus \mathcal{D}_g ) $.
Therefore we have reduced the problem to
the study of antiautomorphisms on
$ \mathcal{D}_e \oplus \mathcal{D}_g $,
which is isomorphic either to
$M_2(\mathbb{R})$ if $ \nu(g) = +1 $,
or to $\mathbb{H}$ if $ \nu(g) = -1 $.

If $ \varphi(J) = +J $, then $\varphi$ is
$\mathcal{D}_e$-semilinear on $\mathcal{D}_g$,
hence there exists $ X \in \mathcal{D}_g $
such that $ \varphi(X) = X $
(and $ \varphi(JX) = -JX $).
Therefore, $\varphi$ is an involution and
it is only compatible with the proper refinement
that splits $\mathcal{D}_g$ as
$ \mathbb{R} X \oplus \mathbb{R} JX $.
It is  straightforward to check the assertions about the isomorphisms classes.
Let us see that involutions that lie
in the same item of the list
are equivalent.
If $ n \geq 4 $, we can always choose $g$ such that
$ \nu(g) = +1 $ and the quadratic form
$ \mu_g(k) := \nu(gk) \nu(g)^{-1} $
is different from $\eta$,
so we can write any involution in (2-a-1) (respectively (2-a-2))
as the tensor product of
an involution on $M_2(\mathbb{R})$ and
an involution on $M_{n/2}(\mathbb{R})$
that lies in (1-a-2) (respectively (1-a-3)),
hence they are all equivalent.

If $ \varphi(J) = -J $, then
$ \varphi \vert_{\mathcal{D}_g} =
\lambda \, \mathrm{id}_{\mathcal{D}_g} $,
where $ \lambda \in \mathcal{D}_e $.
Therefore, $\varphi$ is an involution
if and only if $ \lambda = \pm 1 $ and,
in that case,
every refinement is compatible with $\varphi$.
Again,
we can always choose $g$ such that
$ \nu(g) = +1 $ and
$ \omega(g) = +1 $ (respectively $ \omega(g) = -1 $),
so we can write any involution in (2-a-4) (respectively (2-a-6))
as the tensor product of
an involution on $M_2(\mathbb{R})$ and
an involution on $M_{n/2}(\mathbb{R})$
that lies in (1-a-2),
hence they are all equivalent.

The same arguments work for (2-b) and (2-c),
and also for the case (2-d), which is, in fact, easier
because there is no distinguished involution.

Let us now consider the remaining case (2-e).
Any proper refinement of the grading
has to split $\mathcal{D}_k$,
for all $ k \in T_{[2]} $,
as $ \mathcal{D}_k^{+} \oplus \mathcal{D}_k^{-} $,
where the squares of the elements
in $\mathcal{D}_k^{+}$ (respectively $\mathcal{D}_k^{-}$)
are positive (respectively negative) multiples of $I$.
Also recall from \cite[Remark 21]{Rodrigo-2016} that,
if $ X \in \mathcal{D}_g $,
then there exists a proper refinement of the grading
such that the element $X$ is still homogeneous;
this implies that $\mathcal{D}_{gk}$ splits as
$ X \mathcal{D}_k^{+} \oplus X \mathcal{D}_k^{-} $
for all $ k \in T_{[2]} $.

We have 
$ \varphi( \mathcal{D}_k^{+} ) = \mathcal{D}_k^{+} $ and
$ \varphi( \mathcal{D}_k^{-} ) = \mathcal{D}_k^{-} $
for all $ k \in T_{[2]} $.
Assume that $ \varphi(J) = +J $.
As before, $\varphi$ is
$\mathcal{D}_e$-semilinear on $\mathcal{D}_g$
and there exists $ X \in \mathcal{D}_g $
such that $ \varphi(X) = X $
(and $ \varphi(JX) = -JX $).
This implies that
$ \varphi( X \mathcal{D}_k^{+} ) = X \mathcal{D}_k^{+} $ and
$ \varphi( X \mathcal{D}_k^{-} ) = X \mathcal{D}_k^{-} $
for all $ k \in T_{[2]} $,
that is,
$\varphi$ is an involution and
there is exactly one proper refinement
compatible with $\varphi$.
Assume that $ \varphi(J) = -J $.
Then
$ \varphi \vert_{\mathcal{D}_g} =
\lambda \, \mathrm{id}_{\mathcal{D}_g} $,
where $ \lambda \in \mathcal{D}_e $, thus
$ \varphi \vert_{\mathcal{D}_{gk}} =
\pm \lambda \, \mathrm{id}_{\mathcal{D}_{gk}} $
for all $ k \in T_{[2]} $.
Therefore $\varphi$ is an involution if and only if
$ \lambda = \pm 1 $, and, in that case,
every refinement is compatible with $\varphi$.

Now that we know that every involution is compatible
with at least one proper refinement,
we can use this fact to prove
the rest of the assertions of the theorem
(see Remark \ref{rem:fT}).
Unlike in the previous cases,
in (2-e), if $\psi$ is any isomorphism or equivalence
between two refinements with supports
$ \langle h_1 \rangle \times T_1 $ and
$ \langle h_2 \rangle \times T_2 $,
then $\psi$ will continue to be an isomorphism or equivalence
with respect to the original gradings,
with supports $T_1$ and $T_2$.
Indeed, $\psi$ has to send $h_1$ to $h_2$,
because they are distinguished elements.

Finally, the computation of signature of $\varphi$ can be done similarly to Section \ref{sect:Dim1}
or, alternatively, we can take a compatible refinement and see the signature of the corresponding isomorphism class 
already on the list of Section \ref{sect:Dim1}.
\end{proof}

\section{Classification in the two-dimensional complex case}\label{sect:Dim2Complex}

\begin{example}\label{exam:Inv2f}
Let $ \varepsilon = e^{ 2 \pi i / l } \in \mathbb{C} $
and consider the generalized Pauli Matrices
$ X_a , X_b \in M_l(\mathbb{C}) $
of Figure \ref{fig:Matr}.
Note that
\[
X_a X_b = \varepsilon X_b X_a
\quad\text{and}\quad
X_a^l = X_b^l = I.
\]
Therefore, we can construct a division grading on $M_l(\mathbb{C})$
by the group $ \mathbb{Z}_l \times \mathbb{Z}_l $
if we define the homogeneous component of degree
$ ( \overline{j} , \overline{k} ) $ to be
$ \mathbb{C} X_a^j X_b^k $.
Let $\varphi_A$ and $\varphi_B$ be
the second kind antiautomorphisms on $M_l(\mathbb{C})$
given by $ \varphi_A (X) = A^{-1} X^* A $
and $ \varphi_B (X) = B^{-1} X^* B $,
where $ X^* := \overline{X^T} $
and the matrices $ A,B \in M_l(\mathbb{C}) $
are those of Figure \ref{fig:Matr}.
Since $ A^* = A $ and $ B^* = B $,
both $\varphi_A$ and $\varphi_B$ are involutions.
The signatures of $\varphi_A$ and $\varphi_B$ are,
respectively,
$2$ and $0$ if $l$ is even,
and $1$ and $1$ if $l$ is odd.
Both involutions respect the grading because:
\[
\varphi_A (X_a) = X_a ,
\quad
\varphi_A (X_b) = X_b ;
\qquad
\varphi_B (X_a) = \varepsilon X_a ,
\quad
\varphi_B (X_b) = X_b .
\]
\begin{figure}
\[
X_a =
\begin{pmatrix}
\varepsilon^{l-1} & 0 & 0 & \cdots & 0 \\
0 & \varepsilon^{l-2} & 0 & \cdots & 0 \\
0 & 0 & \varepsilon^{l-3} & \cdots & 0 \\
\vdots & \vdots & \vdots & \ddots & \vdots \\
0 & 0 & 0 & \cdots & 1
\end{pmatrix}
\qquad
X_b =
\begin{pmatrix}
0 & 1 & 0 & \cdots & 0 \\
0 & 0 & 1 & \cdots & 0 \\
\vdots & \vdots & \vdots & \ddots & \vdots \\
0 & 0 & 0 & \cdots & 1 \\
1 & 0 & 0 & \cdots & 0
\end{pmatrix}
\]

\[
A =
\begin{pmatrix}
0 & 0 & \cdots & 1 & 0 \\
\vdots & \vdots & \iddots & \vdots & \vdots \\
0 & 1 & \cdots & 0 & 0 \\
1 & 0 & \cdots & 0 & 0 \\
0 & 0 & \cdots & 0 & 1
\end{pmatrix}
\qquad
B =
\begin{pmatrix}
0 & 0 & 0 & \cdots & 1 \\
\vdots & \vdots & \vdots & \iddots & \vdots \\
0 & 0 & 1 & \cdots & 0 \\
0 & 1 & 0 & \cdots & 0 \\
1 & 0 & 0 & \cdots & 0
\end{pmatrix}
\]
\caption{Matrices in $M_l(\mathbb{C})$
of Example \ref{exam:Inv2f};
$ \varepsilon = e^{ 2 \pi i / l } $.}
\label{fig:Matr}
\end{figure}
We will write:
\begin{itemize}
\item
$ \varphi_l = \varphi_A $,
\item
$ \phi_l = \varphi_B $.
\end{itemize}
\end{example}

\begin{example}
We introduce two involutions on $M_2(\mathbb{C})$
that respect the division grading induced
by the Pauli matrices in the case $l=2$:
\begin{itemize}
\item
Let $\varphi_{\operatorname{(2-f-1-1)}}$ be
the matrix transpose on $M_2(\mathbb{C})$.
It is an involution of the first kind and orthogonal.
\item
Let $\varphi_{\operatorname{(2-f-1-2)}}$ be
the involution on $M_2(\mathbb{C})$
that acts as minus the identity on the matrices of trace zero,
and acts as the identity on the center of $M_2(\mathbb{C})$.
It is an involution of the first kind and symplectic.
\end{itemize}
\end{example}

Let $G$ be an abelian group,
$\mathcal{D}$ a real algebra isomorphic to $M_n(\mathbb{C})$,
and $\Gamma$ a division $G$-grading on $\mathcal{D}$ with homogeneous components of dimension $2$
such that the identity component coincides with the center of $\mathcal{D}$.
$\Gamma$ can be regarded as a grading of the complex algebra $M_n(\mathbb{C})$,
and these gradings are classified in \cite[Theorem 2.15]{EK-2013};
there is one family of equivalence classes: (2-f).
The isomorphism and equivalence classes in this classification remain the same over $\mathbb{R}$,
because the invariants that differentiate them, namely, the pair $(T,\beta)$ and the isomorphism class of $T$ 
respectively, are also preserved by isomorphisms of real algebras.
As always, $T$ is the support of $\Gamma$ and $\beta$ is the alternating bicharacter 
given by the commutation relations, $X_u X_v = \beta(u,v) X_v X_u$ (where $0\ne X_t\in\mathcal{D}_t$ for all $t\in T$), 
but, in contrast with Section \ref{sect:Dim1}, $\beta$ is now $\mathbb{C}$-valued.

Any antiautomorphism $\varphi$ of the $G$-graded algebra $\mathcal{D}$ is an involution,
and satisfies either $\varphi(iI)=+iI$ or $\varphi(iI)=-iI$ (where $I$ is the unity of $\mathcal{D}$).
We classify the pairs $(\Gamma,\varphi)$, up to isomorphism and up to equivalence,
and we give a representative of every equivalence class:

\medskip

(2-f)
The grading $\Gamma$ on $ \mathcal{D} \cong M_n(\mathbb{C}) $ ($ n \geq 1 $) was determined up to isomorphism by $(T,\beta)$,
where $T$ was a subgroup of $G$ isomorphic to $ \mathbb{Z}_{l_1}^2 \times \ldots \times \mathbb{Z}_{l_r}^2 $ ($ l_1 \cdots l_r = n $),
and $\beta$ was a $\mathbb{C}$-valued alternating bicharacter on $T$ such that $ \mathrm{rad}(\beta) = \{ e \} $.
The equivalence class of the grading $\Gamma$ was determined by the isomorphism class of the group $T$,
which we fix henceforth.
Now, in the case $\varphi(iI)=+iI$, we have $ l_1 = \ldots = l_r = 2 $
(so $\beta$ takes values in $ \{ \pm 1 \} \subseteq \mathbb{R}^{\times} $)
and $(\Gamma,\varphi)$ is determined up to isomorphism by $(T,\beta,\eta)$,
where $\eta$ is a quadratic form on $T$ such that $ \beta_{\eta} = \beta $,
and $\eta$ is defined by the equation
\begin{equation}
\varphi(X_t) = \eta(t) X_t
\end{equation}
for all $ X_t \in \mathcal{D}_t $.
These isomorphism classes belong to one of the following two equivalence classes:
\begin{enumerate}
	\item[(1-1)] $ \mathrm{Arf}(\eta) = +1 $ ($ n = 2^m \geq 1 $).\\
	The involution $\varphi$ is of the first kind and orthogonal.\\
	A representative is
	$ \varphi_{\operatorname{(2-f-1-1)}} \otimes_{\mathbb{C}} \dots \otimes_{\mathbb{C}} \varphi_{\operatorname{(2-f-1-1)}} $
	(if $n=1$,  we take $\varphi=\mathrm{id}_\mathbb{C}$).
	\item[(1-2)] $ \mathrm{Arf}(\eta) = -1 $ ($ n = 2^m \geq 2 $).\\
	The involution $\varphi$ is of the first kind and symplectic.\\
	A representative is
	$ \varphi_{\operatorname{(2-f-1-1)}} \otimes_{\mathbb{C}} \dots \otimes_{\mathbb{C}}
	\varphi_{\operatorname{(2-f-1-1)}} \otimes_{\mathbb{C}} \varphi_{\operatorname{(2-f-1-2)}} $.
\end{enumerate}
On the other hand, in the case $\varphi(iI)=-iI$, $(\Gamma,\varphi)$ is determined up to isomorphism by $(T,\beta,S)$,
where $S$ is a subgroup of $T_{[2]}$ of index $1$ or $2$,
and $S$ is defined as
\begin{equation}
S := \{ t \in T_{[2]} \mid
\exists X \in \mathcal{D}_t
\text{ such that } X^2 = +I
\text{ and } \varphi(X) = X \} .
\end{equation}
These isomorphism classes belong to one of the following equivalence classes:
\begin{enumerate}
	\item[(2-0)] $ S = T_{[2]} $ ($ n \geq 1 $).\\
	The involution $\varphi$ is of the second kind and has signature $\sqrt{ \vert T_{[2]} \vert } $.\\
	A representative is
	$ \varphi_{l_1} \otimes_{\mathbb{C}} \dots \otimes_{\mathbb{C}} \varphi_{l_r} $.
	\item[(2-$p$)] $ S \neq T_{[2]} $,
	and any $ t \in T $ of order $2^{p+1}$ satisfies $ t^{2^p} \in S $,
	but there exists $ t \in T $ of order $2^p$ such that $ t^{2^{p-1}} \in T_{[2]} \setminus S $
	($ p \geq 1 $).\\
	The involution $\varphi$ is of the second kind and has signature $0$.\\
	A representative is
	$ \phi_{l_1} \otimes_{\mathbb{C}} \dots \otimes_{\mathbb{C}} \phi_{l_s} \otimes_{\mathbb{C}}
	\varphi_{l_{s+1}} \otimes_{\mathbb{C}} \dots \otimes_{\mathbb{C}} \varphi_{l_r} $,
	provided that $ l_1 \mid l_2 \mid \dots \mid l_r $
	and $2^{p+1}$ divides $l_{s+1}$ but does not divide $l_s$.
\end{enumerate}

\begin{proof}
The case $\varphi(iI)=+iI$ was proved in
\cite[Propositions 2.51 and 2.53]{EK-2013}.
So assume that $\varphi(iI)=-iI$.
For any element $ t \in T $, denote its order by $ o(t) $ and define:
\begin{equation}\label{eq:powers}
\mathcal{D}_t^{[+]} :=
\{ X \in \mathcal{D}_t \mid X^{ o(t) } = +I \}
\quad\text{and}\quad
\mathcal{D}_t^{[-]} :=
\{ X \in \mathcal{D}_t \mid X^{ o(t) } = -I \}.
\end{equation}
Note that,
if $ X \in \mathcal{D}_t^{[+]} $ and
$ \varepsilon \in \mathbb{C} $ is
a primitive $ o(t) $-th root of unity,
then $ \mathcal{D}_t^{[+]} =
\{ X, \allowbreak \varepsilon X , \allowbreak \ldots ,
\allowbreak \varepsilon^{ o(t) - 1 } X \} $,
and similarly for $ \mathcal{D}_t^{[-]} $.
Besides, $ \varphi( \mathcal{D}_t^{[+]} ) = \mathcal{D}_t^{[+]} $
and $ \varphi( \mathcal{D}_t^{[-]} ) = \mathcal{D}_t^{[-]} $;
in particular, $\varphi$ is an involution.
We define the following subsets of $T$,
which are invariants of the isomorphism class of $(\Gamma,\varphi)$:
\begin{equation}
S' := \{ t \in T \mid
\exists X \in \mathcal{D}_t^{[+]}
\text{ such that } \varphi(X) = X \}
\quad\text{and}\quad
S := S' \cap T_{[2]}.
\end{equation}
If $ o(t) $ is odd, then $ t \in S' $,
while if $ o(t) $ is even,
then $ t \in T \setminus S' $ if and only if
there exists $ X \in \mathcal{D}_t^{[-]} $
such that $ \varphi(X) = X $.

Write $T$ as $ U \times V $, where $U$ is
the subgroup of $T$ formed by the elements whose order is a power of $2$,
and $V$ is the subgroup of $T$ formed by the elements of odd order.
We know that $ V \subseteq S' $.
Moreover, if $ u \in U $ and $ v \in V $,
then $ u \in S' $ if and only if $ uv \in S' $,
because $ \beta(u,v) = 1 $.
Finally, if $ u \in U \setminus T_{[2]} $,
then $ u \in S' $ if and only if $ u^2 \in S' $.
Therefore, $S$ determines $S'$.

The restriction of $\beta$ to $ T_{[2]} \times T_{[2]} $
takes values in $ \{ \pm 1 \} $.
Hence, $ u,v \in S $ implies $ uv \in S $,
and also $ u,v \in T_{[2]} \setminus S $ implies $ uv \in S $.
Therefore, $S$ is a subgroup of $T_{[2]}$ of index $1$ or $2$.

We know, for example from \cite[Equation (2.6)]{EK-2013},
that we can write $T$ as follows:
\begin{equation}\label{eq:SymBas}
T =
\langle a_1 \rangle \times \langle b_1 \rangle
\times \ldots \times
\langle a_r \rangle \times \langle b_r \rangle,
\end{equation}
where
$ a_i , b_i \in T $,
$ \langle a_i \rangle \times \langle b_i \rangle
\cong \mathbb{Z}_{l_i}^2 $,
$l_i$ is a power of a prime,
$ \beta(a_i,b_i) = \beta(b_i,a_i)^{-1} = e^{ 2 \pi i / l_i } $,
and the value of $\beta$ on all other pairs is $1$.
We claim that $(T,\beta,S)$ determines
$(\Gamma,\varphi)$ up to isomorphism.
We can pick $ X_{a_i} \in \mathcal{D}_{a_i} $
such that $ \varphi(X_{a_i}) = X_{a_i} $,
and either $ X_{a_i} \in \mathcal{D}_{a_i}^{[+]} $
if $ a_i \in S' $,
or $ X_{a_i} \in \mathcal{D}_{a_i}^{[-]} $
if $ a_i \in T \setminus S' $.
We pick $ X_{b_i} \in \mathcal{D}_{b_i} $ in the same way.
The elements $X_{a_i}$, $X_{b_i}$ generate $\mathcal{D}$,
with defining relations of two kinds: the commutation relations are determined by $\beta$
and the power relations are determined by $S$ through Equation \eqref{eq:powers}.
This proves the claim;
in fact, the isomorphism can be chosen
to be an isomorphism of complex algebras.
Conversely, let us find an involution $\varphi$
for a given subgroup $S$ of $T_{[2]}$ of index $1$ or $2$.
Thanks to Lemma \ref{lem:SignCompl},
it is enough to construct it for every factor 
$ \langle a_i \rangle \times \langle b_i \rangle $ of $T$,
but we have already done it in Example \ref{exam:Inv2f}.

Let us see that, for a fixed $ p \geq 1 $, all the involutions that lie
in (2-f-2-$p$) are equivalent.
In fact, we will show that
$ a_1 , \allowbreak b_1 , \allowbreak
\ldots , \allowbreak a_r , \allowbreak b_r $
in Equation \eqref{eq:SymBas} may be chosen
so that they also satisfy:
$ a_i , b_i \in S' $ for all $i$,
except in the case $ l_i = 2^{m_i} \leq 2^p $,
when $ a_i \in S' $ but $ b_i \in T \setminus S' $.
We can follow the same induction process as the one leading to
\cite[Equation (2.6)]{EK-2013},
until we arrive to a situation in which $T$ is a $2$-group
and there are elements in $T$ of maximal order, $2^p$,
that do not belong to $S'$.
Rearranging, we may assume that
$ l_1 = 2^{m_1} \geq l_2 = 2^{m_2}
\geq \ldots \geq l_r = 2^{m_r} $.
If $ r = 1 $, the statement is clear,
so suppose that $ r \geq 2 $.
Then we choose the next $a,b$ in the following way.

We want to take $ a,b \in T $ such that
$ o(a) = o(b) = 2^{m_1} $,
$ \beta(a,b) = e^{ 2 \pi i / l_1 } $,
$ a \in S' $, $ b \in T \setminus S' $,
and such that there are elements in
$ \langle a,b \rangle^{\perp} $
of maximal order, $2^{m_2}$, that do not belong to $S'$,
because then $ T = \langle a,b \rangle
\times \langle a,b \rangle^{\perp} $
and we will be able to continue the induction process with
$ \langle a,b \rangle^{\perp} $.
We know the existence of a decomposition $ T =
\langle \tilde{a}_1 \rangle \times
\langle \tilde{b}_1 \rangle
\times \ldots \times
\langle \tilde{a}_r \rangle \times
\langle \tilde{b}_r \rangle $
as in Equation \eqref{eq:SymBas},
but we cannot assure that $ \tilde{a}_i \in S' $
and $ \tilde{b}_i \in T \setminus S' $.
Without loss of generality,
$ \tilde{a}_1 , \tilde{a}_2 \in S' $
and $ \tilde{b}_1 \in T \setminus S' $,
hence $ \tilde{a}_1 \tilde{a}_2 \in S' $.
If $ \tilde{b}_2 \in T \setminus S' $,
simply take $ a = \tilde{a}_1 $
and $ b = \tilde{b}_1 $.
If $ \tilde{b}_2 \in S' $,
take $ a = \tilde{a}_1 \tilde{a}_2 $
and $ b = \tilde{b}_1 $, and note that
$ \tilde{b}_1^{ l_1 / l_2 } \tilde{b}_2^{-1} $
has order $2^{m_2}$ and belongs both to
$ \langle a,b \rangle^{\perp} $
and to $ T \setminus S' $.

Finally, the computation of signature is analogous to Section \ref{sect:Dim1}, but using 
Example \ref{exam:Inv2f} instead of Example \ref{exam:GradDim1} and
Lemma \ref{lem:SignCompl} instead of Lemma \ref{lem:SignReal}.
For involutions in (2-f-2-$p$), we pick up a zero factor. 
For involutions in (2-f-2-0), we may assume that $ l_1 , \ldots , l_s $ are even and $ l_{s+1} , \ldots , l_r $ are odd,
then $s$ is the number of factors $2$, so the signature equals $2^s=\sqrt{ \vert T_{[2]} \vert }$.
\end{proof}

\begin{remark}\label{rem:all_inv_2}
Consider an involution $\varphi$ of the second kind
on the graded algebra $\mathcal{D}$.
By Lemma \ref{lem:InnAut}, all such involutions 
can be obtained as $ \mathrm{Int}(X_u) \circ \varphi $ where $u$ runs through $T$.
Since $\beta$ is nondegenerate, it is easy to see that 
$u\in T^{[2]}$ (recall Notation \ref{nota:subgroups})
if and only if $ \beta(u,v) = 1 $ for all $ v \in T_{[2]} $.
Therefore, $ \mathrm{Int}(X_u) \circ \varphi $ and $\varphi$
are in the same isomorphism class
if and only if $u\in T^{[2]}$, because 
$ ( \mathrm{Int}(X_u) \circ \varphi )(X_v)
= \beta(u,v) \varphi(X_v) $.
Now assume that $\varphi$ lies in (2-f-2-0).
We have just shown that $ \mathrm{Int}(X_u) \circ \varphi $
lies in (2-f-2-0) if and only if $u$ is a square in $T$.
Now we claim that $ \mathrm{Int}(X_u) \circ \varphi $
lies in (2-f-2-$p$) ($ p \geq 1 $) if and only if
$ u T_{[2^p]} $ is a square in $ T / T_{[2^p]} $ but
$ u T_{[2^{p-1}]} $ is not a square in $ T / T_{[2^{p-1}]} $.
Indeed, using the nondegeneracy of $\beta$ (or explicitly using its values on the pairs of generators in 
Equation \eqref{eq:SymBas}), it is straightforward to show that
$ u T_{[2^p]} $ is a square in $ T / T_{[2^p]} $
if and only if $\beta(u,v^{2^p})=1$ for all $v\in T_{[2^{p+1}]}$.
%
%
\end{remark}

\section{Classification in the four-dimensional case}\label{sect:Dim4}

Let $G$ be an abelian group,
$\mathcal{D}$ a finite-dimensional simple real (associative) algebra,
and $\Gamma$ a division $G$-grading on $\mathcal{D}$ with homogeneous components of dimension $4$.
We can apply the Double Centralizer Theorem
(see for example \cite[Theorem 4.7]{Jacobson-1989})
to the identity component $\mathcal{D}_e$,
which is isomorphic to $\mathbb{H}$,
to conclude that $\mathcal{D}$ is isomorphic,
as a graded algebra, to $ \mathcal{D}_e
\otimes_{\mathbb{R}} C_{\mathcal{D}} (\mathcal{D}_e) $
(see \cite[Theorem 19]{Rodrigo-2016} for more details).
Note that $ C_{\mathcal{D}} (\mathcal{D}_e) $ is again a
finite-dimensional simple real 
graded-division algebra,
but with homogeneous components of dimension $1$.
Any antiautomorphism $\varphi$
of the $G$-graded algebra $\mathcal{D}$ is
the tensor product of its restrictions to $\mathcal{D}_e$
and to $ C_{\mathcal{D}} (\mathcal{D}_e) $.

The following result is well known and easily follows from Skolem--Noether Theorem.

\begin{proposition}\label{prop:InvH}
Any antiautomorphism $\varphi$ of the real algebra $\mathbb{H}$ can be written as $ \varphi(X) = A^{-1} \overline{X} A $,
for some $ A = a+bi+cj+dk \in \mathbb{H}^{\times} $.
So $\varphi$ is an involution if and only if either $b=c=d=0$ or $a=0$.
In the first case $\varphi$ is symplectic with signature $1$ ($\varphi$ is the conjugation),
while in the second $\varphi$ is orthogonal.
\qed
\end{proposition}

For us, this means the following: if $\Gamma$ is the trivial grading on $\mathbb{H}$, 
then there are exactly two isomorphism classes of pairs $(\Gamma,\varphi)$, where $\varphi$ 
is an involution, and they coincide with the equivalence classes.
When $\mathbb{H}$ is endowed with the trivial grading, we denote:
\begin{itemize}
\item
Let $\varphi_{\operatorname{(3-b-1)}}$ be
the conjugation on $\mathbb{H}$, that is,
$ \varphi_{\operatorname{(3-b-1)}} = \varphi_{\operatorname{(1-b-1)}} $.
\item
Let $\varphi_{\operatorname{(3-b-4)}}$ be
a representative of the orthogonal equivalence class, for example,
$ \varphi_{\operatorname{(3-b-4)}} = \varphi_{\operatorname{(1-b-3)}} $.
\end{itemize}

Now, the classification of pairs $(\Gamma,\varphi)$, where $\Gamma$ is a division grading on 
$\mathcal{D}$ as above and $\varphi$ is an involution,
is easily obtained from
Proposition \ref{prop:InvH} and Section \ref{sect:Dim1}.
Therefore, the isomorphism classes are in two-to-one correspondence with
the quadratic forms $\eta$ on $T$ such that $ \beta_{\eta} = \beta $,
where $T$ is the support of $\Gamma$ and
$\beta$ is the alternating bicharacter on $T$
given by the commutation relations
in the centralizer of the identity component.
Each quadratic form $\eta$ corresponds to two isomorphism classes,
one for every class of involutions on
$ \mathcal{D}_e \cong \mathbb{H} $,
by means of the equation:
\begin{equation}
\varphi(X_t) = \eta(t) X_t
\end{equation}
for all $ X_t \in \mathcal{D}_t \cap C_{\mathcal{D}} (\mathcal{D}_e) $.
Note that, if we have to compute the signature of $\varphi$,
we take a compatible refinement
and check the signature of
the corresponding isomorphism class
in the list of Section \ref{sect:Dim1}.
We compile the classification,
up to isomorphism and up to equivalence,
and we give a representative of every equivalence class,
to serve as a reference:

\medskip

(3-a)
The grading $\Gamma$ on $ \mathcal{D} \cong M_n(\mathbb{R}) $ ($ n = 2^m \geq 4 $) was determined up to isomorphism by $(T,\mu)$,
where $T$ was a subgroup of $G$ isomorphic to $\mathbb{Z}_2^{2m-2}$,
and $\mu$ was a quadratic form on $T$ such that $\beta_{\mu}$ had type I and $ \mathrm{Arf}(\mu) = -1 $.
Now, if $\varphi$ is the conjugation on $\mathcal{D}_e$, 
$(\Gamma,\varphi)$ is determined up to isomorphism by $(T,\mu,\eta)$,
where $\eta$ is a quadratic form on $T$ such that $ \beta_{\eta} = \beta_{\mu} $.
These isomorphism classes belong to one of the following three equivalence classes:
\begin{enumerate}
	\item[(1)] $ \eta = \mu $ ($ n = 2^m \geq 4 $).\\
	The involution $\varphi$ is orthogonal with signature $n$.\\
	A representative is
	$ \varphi_{\operatorname{(3-b-1)}} \otimes
	\varphi_* \otimes \dots \otimes \varphi_* \otimes
	\varphi_{\operatorname{(1-b-1)}} $.
	\item[(2)] $ \mathrm{Arf}(\eta) = -1 $ but $ \eta \neq \mu $ ($ n = 2^m \geq 8 $).\\
	The involution $\varphi$ is orthogonal with signature $0$.\\
	A representative is
	$ \varphi_{\operatorname{(3-b-1)}} \otimes
	\varphi_* \otimes \dots \otimes \varphi_* \otimes
	\varphi_{\operatorname{(1-a-2)}} \otimes \varphi_{\operatorname{(1-b-1)}} $.
	\item[(3)] $ \mathrm{Arf}(\eta) = +1 $ ($ n = 2^m \geq 4 $).\\
	The involution $\varphi$ is symplectic.\\
	A representative is
	$ \varphi_{\operatorname{(3-b-1)}} \otimes
	\varphi_* \otimes \dots \otimes \varphi_* \otimes
	\varphi_{\operatorname{(1-b-3)}} $.
\end{enumerate}
On the other hand, if $\varphi$ is orthogonal on $\mathcal{D}_e$, 
$(\Gamma,\varphi)$ is determined up to isomorphism by $(T,\mu,\eta)$,
where again $\eta$ is a quadratic form on $T$ such that $ \beta_{\eta} = \beta_{\mu} $.
These isomorphism classes belong to one of the following three equivalence classes:
\begin{enumerate}
	\item[(4)] $ \eta = \mu $ ($ n = 2^m \geq 4 $).\\
	The involution $\varphi$ is symplectic.\\
	A representative is
	$ \varphi_{\operatorname{(3-b-4)}} \otimes
	\varphi_* \otimes \dots \otimes \varphi_* \otimes
	\varphi_{\operatorname{(1-b-1)}} $.
	\item[(5)] $ \mathrm{Arf}(\eta) = -1 $ but $ \eta \neq \mu $ ($ n = 2^m \geq 8 $).\\
	The involution $\varphi$ is symplectic.\\
	A representative is
	$ \varphi_{\operatorname{(3-b-4)}} \otimes
	\varphi_* \otimes \dots \otimes \varphi_* \otimes
	\varphi_{\operatorname{(1-a-2)}} \otimes \varphi_{\operatorname{(1-b-1)}} $.
	\item[(6)] $ \mathrm{Arf}(\eta) = +1 $ ($ n = 2^m \geq 4 $).\\
	The involution $\varphi$ is orthogonal with signature $0$.\\
	A representative is
	$ \varphi_{\operatorname{(3-b-4)}} \otimes
	\varphi_* \otimes \dots \otimes \varphi_* \otimes
	\varphi_{\operatorname{(1-b-3)}} $.
\end{enumerate}

\medskip

(3-b)
The grading $\Gamma$ on $ \mathcal{D} \cong M_{n/2}(\mathbb{H}) $ ($ n = 2^m \geq 2 $) was determined up to isomorphism by $(T,\mu)$,
where $T$ was a subgroup of $G$ isomorphic to $\mathbb{Z}_2^{2m-2}$,
and $\mu$ was a quadratic form on $T$ such that $\beta_{\mu}$ had type I and $ \mathrm{Arf}(\mu) = +1 $.
Now, if $\varphi$ is the conjugation on $\mathcal{D}_e$, 
$(\Gamma,\varphi)$ is determined up to isomorphism by $(T,\mu,\eta)$,
where $\eta$ is a quadratic form on $T$ such that $ \beta_{\eta} = \beta_{\mu} $.
These isomorphism classes belong to one of the following three equivalence classes:
\begin{enumerate}
	\item[(1)] $ \eta = \mu $ ($ n = 2^m \geq 2 $).\\
	The involution $\varphi$ is symplectic with signature $n/2$.\\
	A representative is
	$ \varphi_{\operatorname{(3-b-1)}} \otimes
	\varphi_* \otimes \dots \otimes \varphi_* $.
	\item[(2)] $ \mathrm{Arf}(\eta) = +1 $ but $ \eta \neq \mu $ ($ n = 2^m \geq 4 $).\\
	The involution $\varphi$ is symplectic with signature $0$.\\
	A representative is
	$ \varphi_{\operatorname{(3-b-1)}} \otimes
	\varphi_* \otimes \dots \otimes \varphi_* \otimes
	\varphi_{\operatorname{(1-a-2)}} $.
	\item[(3)] $ \mathrm{Arf}(\eta) = -1 $ ($ n = 2^m \geq 4 $).\\
	The involution $\varphi$ is orthogonal.\\
	A representative is
	$ \varphi_{\operatorname{(3-b-1)}} \otimes
	\varphi_* \otimes \dots \otimes \varphi_* \otimes
	\varphi_{\operatorname{(1-a-3)}} $.
\end{enumerate}
On the other hand, if $\varphi$ is orthogonal on $\mathcal{D}_e$, 
$(\Gamma,\varphi)$ is determined up to isomorphism by $(T,\mu,\eta)$,
where again $\eta$ is a quadratic form on $T$ such that $ \beta_{\eta} = \beta_{\mu} $.
These isomorphism classes belong to one of the following three equivalence classes:
\begin{enumerate}
	\item[(4)] $ \eta = \mu $ ($ n = 2^m \geq 2 $).\\
	The involution $\varphi$ is orthogonal.\\
	A representative is
	$ \varphi_{\operatorname{(3-b-4)}} \otimes
	\varphi_* \otimes \dots \otimes \varphi_* $.
	\item[(5)] $ \mathrm{Arf}(\eta) = +1 $ but $ \eta \neq \mu $ ($ n = 2^m \geq 4 $).\\
	The involution $\varphi$ is orthogonal.\\
	A representative is
	$ \varphi_{\operatorname{(3-b-4)}} \otimes
	\varphi_* \otimes \dots \otimes \varphi_* \otimes
	\varphi_{\operatorname{(1-a-2)}} $.
	\item[(6)] $ \mathrm{Arf}(\eta) = -1 $ ($ n = 2^m \geq 4 $).\\
	The involution $\varphi$ is symplectic with signature $0$.\\
	A representative is
	$ \varphi_{\operatorname{(3-b-4)}} \otimes
	\varphi_* \otimes \dots \otimes \varphi_* \otimes
	\varphi_{\operatorname{(1-a-3)}} $.
\end{enumerate}

\medskip

(3-c)
The grading $\Gamma$ on $ \mathcal{D} \cong M_n(\mathbb{C}) $ ($ n = 2^m \geq 2 $) was determined up to isomorphism by $(T,\mu)$,
where $T$ was a subgroup of $G$ isomorphic to $\mathbb{Z}_2^{2m-1}$,
and $\mu$ was a quadratic form on $T$ such that $\beta := \beta_{\mu}$ had type II and $ \mu(f_{\beta}) = -1 $.
Now, if $\varphi$ is the conjugation on $\mathcal{D}_e$, 
$(\Gamma,\varphi)$ is determined up to isomorphism by $(T,\mu,\eta)$,
where $\eta$ is a quadratic form on $T$ such that $ \beta_{\eta} = \beta $.
These isomorphism classes belong to one of the following four equivalence classes:
\begin{enumerate}
	\item[(1)] $ \eta = \mu $ ($ n = 2^m \geq 2 $).\\
	The involution $\varphi$ is of the second kind and has signature $n$.\\
	A representative is
	$ \varphi_{\operatorname{(3-b-1)}} \otimes
	\varphi_* \otimes \dots \otimes \varphi_* \otimes
	\varphi_{\operatorname{(1-c-1)}} $.
	\item[(2)] $ \eta(f_{\beta}) = -1 $ but $ \eta \neq \mu $ ($ n = 2^m \geq 4 $).\\
	The involution $\varphi$ is of the second kind and has signature $0$.\\
	A representative is
	$ \varphi_{\operatorname{(3-b-1)}} \otimes
	\varphi_* \otimes \dots \otimes \varphi_* \otimes
	\varphi_{\operatorname{(1-a-2)}} \otimes \varphi_{\operatorname{(1-c-1)}} $.
	\item[(3)] $ \eta(f_{\beta}) = +1 $ and $ \mathrm{Arf}(\eta) = +1 $ ($ n = 2^m \geq 2 $).\\
	The involution $\varphi$ is of the first kind and symplectic.\\
	A representative is
	$ \varphi_{\operatorname{(3-b-1)}} \otimes
	\varphi_* \otimes \dots \otimes \varphi_* \otimes
	\varphi_{\operatorname{(1-c-3)}} $.
	\item[(4)] $ \eta(f_{\beta}) = +1 $ and $ \mathrm{Arf}(\eta) = -1 $ ($ n = 2^m \geq 4 $).\\
	The involution $\varphi$ is of the first kind and orthogonal.\\
	A representative is
	$ \varphi_{\operatorname{(3-b-1)}} \otimes
	\varphi_* \otimes \dots \otimes \varphi_* \otimes
	\varphi_{\operatorname{(1-a-3)}} \otimes \varphi_{\operatorname{(1-c-3)}} $.
\end{enumerate}
On the other hand, if $\varphi$ is orthogonal in $\mathcal{D}_e$, 
$(\Gamma,\varphi)$ is determined up to isomorphism by $(T,\mu,\eta)$,
where again $\eta$ is a quadratic form on $T$ such that $ \beta_{\eta} = \beta $.
These isomorphism classes belong to one of the following four equivalence classes:
\begin{enumerate}
	\item[(5)] $ \eta = \mu $ ($ n = 2^m \geq 2 $).\\
	The involution $\varphi$ is of the second kind and has signature $0$.\\
	A representative is
	$ \varphi_{\operatorname{(3-b-4)}} \otimes
	\varphi_* \otimes \dots \otimes \varphi_* \otimes
	\varphi_{\operatorname{(1-c-1)}} $.
	\item[(6)] $ \eta(f_{\beta}) = -1 $ but $ \eta \neq \mu $ ($ n = 2^m \geq 4 $).\\
	The involution $\varphi$ is of the second kind and has signature $0$.\\
	A representative is
	$ \varphi_{\operatorname{(3-b-4)}} \otimes
	\varphi_* \otimes \dots \otimes \varphi_* \otimes
	\varphi_{\operatorname{(1-a-2)}} \otimes \varphi_{\operatorname{(1-c-1)}} $.
	\item[(7)] $ \eta(f_{\beta}) = +1 $ and $ \mathrm{Arf}(\eta) = +1 $ ($ n = 2^m \geq 2 $).\\
	The involution $\varphi$ is of the first kind and orthogonal.\\
	A representative is
	$ \varphi_{\operatorname{(3-b-4)}} \otimes
	\varphi_* \otimes \dots \otimes \varphi_* \otimes
	\varphi_{\operatorname{(1-c-3)}} $.
	\item[(8)] $ \eta(f_{\beta}) = +1 $ and $ \mathrm{Arf}(\eta) = -1 $ ($ n = 2^m \geq 4 $).\\
	The involution $\varphi$ is of the first kind and symplectic.\\
	A representative is
	$ \varphi_{\operatorname{(3-b-4)}} \otimes
	\varphi_* \otimes \dots \otimes \varphi_* \otimes
	\varphi_{\operatorname{(1-a-3)}} \otimes \varphi_{\operatorname{(1-c-3)}} $.
\end{enumerate}

\medskip

(3-d)
The grading $\Gamma$ on $ \mathcal{D} \cong M_n(\mathbb{C}) $ ($ n = 2^m \geq 4 $) was determined up to isomorphism by $(T,\beta,\mu)$,
where $T$ was a subgroup of $G$ isomorphic to $ \mathbb{Z}_2^{2m-3} \times \mathbb{Z}_4 $,
$\beta$ was an alternating bicharacter on $T$ of type II,
and $\mu$ was a quadratic form on $T_{[2]}$ such that $ \beta_{\mu} = \beta \vert_{ T_{[2]} \times T_{[2]} } $ and $ \mu(f_T) = -1 $.
Now, if $\varphi$ is the conjugation on $\mathcal{D}_e$, $(\Gamma,\varphi)$ is determined up to isomorphism by $(T,\beta,\mu,\eta)$,
where $\eta$ is a quadratic form on $T$ such that $ \beta_{\eta} = \beta $ (so $ \eta(f_T) = +1 $).
These isomorphism classes belong to one of the following four equivalence classes:
\begin{enumerate}
	\item[(1)] $ \eta( \mathrm{rad}'(\beta) ) = \{ +1 \} $ and $ \mathrm{Arf}(\eta) = +1 $ ($ n = 2^m \geq 4 $).\\
	The involution $\varphi$ is of the first kind and symplectic.\\
	A representative is
	$ \varphi_{\operatorname{(3-b-1)}} \otimes
	\varphi_* \otimes \dots \otimes \varphi_* \otimes
	\varphi_{\operatorname{(1-d-1)}} $.
	\item[(2)] $ \eta( \mathrm{rad}'(\beta) ) = \{ +1 \} $ and $ \mathrm{Arf}(\eta) = -1 $ ($ n = 2^m \geq 8 $).\\
	The involution $\varphi$ is of the first kind and orthogonal.\\
	A representative is
	$ \varphi_{\operatorname{(3-b-1)}} \otimes
	\varphi_* \otimes \dots \otimes \varphi_* \otimes
	\varphi_{\operatorname{(1-a-3)}} \otimes \varphi_{\operatorname{(1-d-1)}} $.
	\item[(3)] $ \eta( \mathrm{rad}'(\beta) ) = \{ -1 \} $ and $ \mathrm{Arf}(\eta) = +1 $ ($ n = 2^m \geq 4 $).\\
	The involution $\varphi$ is of the first kind and symplectic.\\
	A representative is
	$ \varphi_{\operatorname{(3-b-1)}} \otimes
	\varphi_* \otimes \dots \otimes \varphi_* \otimes
	\varphi_{\operatorname{(1-d-3)}} $.
	\item[(4)] $ \eta( \mathrm{rad}'(\beta) ) = \{ -1 \} $ and $ \mathrm{Arf}(\eta) = -1 $ ($ n = 2^m \geq 4 $).\\
	The involution $\varphi$ is of the first kind and orthogonal.\\
	A representative is
	$ \varphi_{\operatorname{(3-b-1)}} \otimes
	\varphi_* \otimes \dots \otimes \varphi_* \otimes
	\varphi_{\operatorname{(1-d-4)}} $.
\end{enumerate}
On the other hand, if $\varphi$ is orthogonal on $\mathcal{D}_e$, 
$(\Gamma,\varphi)$ is determined up to isomorphism by $(T,\beta,\mu,\eta)$,
where again $\eta$ is a quadratic form on $T$ such that $ \beta_{\eta} = \beta $ (so $ \eta(f_T) = +1 $).
These isomorphism classes belong to one of the following four equivalence classes:
\begin{enumerate}
	\item[(5)] $ \eta( \mathrm{rad}'(\beta) ) = \{ +1 \} $ and $ \mathrm{Arf}(\eta) = +1 $ ($ n = 2^m \geq 4 $).\\
	The involution $\varphi$ is of the first kind and orthogonal.\\
	A representative is
	$ \varphi_{\operatorname{(3-b-4)}} \otimes
	\varphi_* \otimes \dots \otimes \varphi_* \otimes
	\varphi_{\operatorname{(1-d-1)}} $.
	\item[(6)] $ \eta( \mathrm{rad}'(\beta) ) = \{ +1 \} $ and $ \mathrm{Arf}(\eta) = -1 $ ($ n = 2^m \geq 8 $).\\
	The involution $\varphi$ is of the first kind and symplectic.\\
	A representative is
	$ \varphi_{\operatorname{(3-b-4)}} \otimes
	\varphi_* \otimes \dots \otimes \varphi_* \otimes
	\varphi_{\operatorname{(1-a-3)}} \otimes \varphi_{\operatorname{(1-d-1)}} $.
	\item[(7)] $ \eta( \mathrm{rad}'(\beta) ) = \{ -1 \} $ and $ \mathrm{Arf}(\eta) = +1 $ ($ n = 2^m \geq 4 $).\\
	The involution $\varphi$ is of the first kind and orthogonal.\\
	A representative is
	$ \varphi_{\operatorname{(3-b-4)}} \otimes
	\varphi_* \otimes \dots \otimes \varphi_* \otimes
	\varphi_{\operatorname{(1-d-3)}} $.
	\item[(8)] $ \eta( \mathrm{rad}'(\beta) ) = \{ -1 \} $ and $ \mathrm{Arf}(\eta) = -1 $ ($ n = 2^m \geq 4 $).\\
	The involution $\varphi$ is of the first kind and symplectic.\\
	A representative is
	$ \varphi_{\operatorname{(3-b-4)}} \otimes
	\varphi_* \otimes \dots \otimes \varphi_* \otimes
	\varphi_{\operatorname{(1-d-4)}} $.
\end{enumerate}

\section{Semisimple algebras with involution}\label{sect:Semisimple}

The results of the previous sections can be extended 
to finite-dimensional real algebras that are not necessarily simple,
but do not have nontrivial ideals preserved by the involution.
As mentioned in the Introduction, our purpose is to apply these results for the classification of gradings
on classical real Lie algebras in a forthcoming article \cite{BKR-2017},
so here we restrict ourselves to the situation relevant for that application.

Let $G$ be an abelian group,
$\mathcal{D}$ a finite-dimensional
\emph{non}-simple real (associative) algebra
whose center has dimension $2$,
and $\Gamma$ a division $G$-grading on $\mathcal{D}$.
Recall from Section \ref{sect:GradQF} that this implies that
$\mathcal{D}$ is the direct product of two central simple algebras over $\mathbb{R}$.
Let $\varphi$ be a second kind involution
on the $G$-graded algebra $\mathcal{D}$.
We want to classify the pairs $(\Gamma,\varphi)$,
up to isomorphism
(but not up to equivalence).
In fact, we can repeat the arguments
in \cite{Rodrigo-2016} and in the previous sections,
because they do not depend on
the simplicity of the underlying algebra.

\medskip

Let us start by considering the grading $\Gamma$ and
disregarding the involution $\varphi$.
As always, we denote by $T$ the support of $\Gamma$,
by $K$ the support of the centralizer of the identity component,
and by $ \beta : K \times K \rightarrow \{ \pm 1 \} $
the alternating bicharacter given by
the commutation relations in
the centralizer of the identity component.
Also, if the homogeneous components have dimension $2$,
we write $ \mathcal{D}_e = \mathbb{R}I \oplus \mathbb{R}J $
($ \cong \mathbb{C} $),
where $I$ is the unity of $\mathcal{D}$ and $J^2=-I$.
By \cite[Proposition 20]{Rodrigo-2016},
if the homogeneous components have dimension $2$ or $4$,
then there exists a proper refinement of the grading.

The existence of a second kind involution $\varphi$
prevents $T$ from having a factor $\mathbb{Z}_4$,
in other words, $T$ is an elementary abelian $2$-group.
Indeed, Remark \ref{rem:fT} can be invoked
if the homogeneous components have dimension $1$.
As in Section \ref{sect:Dim4}, the case of dimension $4$ reduces to dimension $1$ 
using the Double Centralizer Theorem
(note that \cite[Theorem 4.7]{Jacobson-1989}
does not require the ambient algebra to be simple).
Finally, in the case where the homogeneous components have dimension $2$,
if there existed an element $ g \in T \setminus K $ of order $4$, then,
by \cite[Remark 21]{Rodrigo-2016},
any $ 0 \neq X,X' \in \mathcal{D}_g $ would satisfy
$ X^2 \in Z(\mathcal{D}) $ and
$ (X')^2 \in \mathbb{R}_{ > 0 } X^2 $,
so $ \varphi(X^2) = \varphi(X)^2 \in \mathbb{R}_{ > 0 } X^2 $
would give us a contradiction with
$\varphi$ being of the second kind.

\medskip

Looking at the list in Section \ref{sect:GradQF}, we see 
that $\mathcal{D}$ must be isomorphic to
either $ M_n(\mathbb{R}) \times M_n(\mathbb{R}) $
($ n = 2^m \geq 1 $)
or $ M_{n/2}(\mathbb{H}) \times M_{n/2}(\mathbb{H}) $
($ n = 2^m \geq 2 $),
both with a grading whose support is an elementary $2$-group of rank $2m+1$, $2m$ or $2m-1$,
according to the homogeneous components being of dimension $1$, $2$ or $4$ respectively.

\begin{itemize}
\item
If the homogeneous components have dimension $1$,
then $(\Gamma,\varphi)$ is determined up to isomorphism
by $(T,\mu,\eta)$, where
$T$ is a subgroup of $G$ isomorphic to $\mathbb{Z}_2^{2m+1}$,
$\mu$ is a quadratic form on $T$ such that
$\beta := \beta_{\mu}$ has type II
and $ \mu(f_{\beta}) = +1 $,
and $\eta$ is a quadratic form on $T$ such that
$ \beta_{\eta} = \beta $ and $ \eta(f_{\beta}) = -1 $.
\\
If $ \mathrm{Arf}(\mu) = +1 $, then $ \mathcal{D}
\cong M_n(\mathbb{R}) \times M_n(\mathbb{R}) $ ($ n = 2^m \geq 1 $),
whereas if $ \mathrm{Arf}(\mu) = -1 $, then $ \mathcal{D}
\cong M_{n/2}(\mathbb{H}) \times M_{n/2}(\mathbb{H}) $ ($ n = 2^m \geq 2 $).
\item
If the homogeneous components have dimension $4$,
then the classification is again reduced
to the case of dimension $1$ (see Section \ref{sect:Dim4}).
\item
If the homogeneous components have dimension $2$,
then the grading $\Gamma$ is determined up to isomorphism
by $(T,K,\nu)$,
where $T$ is a subgroup of $G$
isomorphic to $\mathbb{Z}_2^{2m}$,
$K$ is a subgroup of $T$ of index $2$,
and $\nu$ is a nice map on $ T \setminus K $ such that
$\beta := \beta_{\nu}$ has type II
and $ \nu(f_{\beta}) = +1 $.
\\
If $ \mathrm{Arf}(\nu) = +1 $, then $ \mathcal{D}
\cong M_n(\mathbb{R}) \times M_n(\mathbb{R}) $ ($ n = 2^m \geq 2 $),
whereas if $ \mathrm{Arf}(\nu) = -1 $, then $ \mathcal{D}
\cong M_{n/2}(\mathbb{H}) \times M_{n/2}(\mathbb{H}) $ ($ n = 2^m \geq 2 $).
\\
Now, in the case $\varphi(J)=+J$,
$(\Gamma,\varphi)$ is determined up to isomorphism
by $(T,K,\nu,\eta)$,
where $\eta$ is a quadratic form on $K$ such that
$ \beta_{\eta} = \beta $ and $ \eta(f_{\beta}) = -1 $.
On the other hand, in the case $\varphi(J)=-J$,
$(\Gamma,\varphi)$ is determined up to isomorphism
by $ ( T , \allowbreak K , \allowbreak \nu ,
\allowbreak \omega ) $,
where $\omega$ is a nice map on $ T \setminus K $
such that $ \beta_{\omega} = \beta $
and $ \omega(f_{\beta}) = -1 $.
\end{itemize}

\section{Distinguished involutions}\label{sect:Distinguished}

Let $\mathcal{D}$ be as in Sections \ref{sect:Dim1}, \ref{sect:Dim2NonComplex}, \ref{sect:Dim2Complex} or 
\ref{sect:Dim4}, that is, a finite-dimensional simple real algebra with a division grading $\Gamma$ 
by an abelian group $G$ such that $\mathcal{D}$ admits an involution as a graded algebra. 
Let $T$ be the support of $\Gamma$.

We already observed (see Remarks \ref{rem:all_inv_1} and \ref{rem:all_inv_2}) that, 
given one such involution $\varphi$, we can obtain all involutions 
(of the same kind in the case $Z(\mathcal{D})=\mathbb{C}$) as 
$\mathrm{Int}(X)\circ\varphi$, where $X$ runs through nonzero homogeneous elements of $\mathcal{D}$.
Over an algebraically closed field such as $\mathbb{C}$, which appears in this paper in 
Section \ref{sect:Dim2Complex} when $\varphi$ is of the first kind, there is no special choice of $\varphi$.
Over the field $\mathbb{R}$, however, we conclude from our results that there is often a special choice, 
which we refer to as a \emph{distinguished involution}.
Here we collect some of the properties of these involutions.

First assume that $T$ is an elementary $2$-group and, if the identity component $\mathcal{D}_e$
has dimension $2$, it does not coincide with $Z(\mathcal{D})$. Then, looking at the lists in Sections 
\ref{sect:Dim1}, \ref{sect:Dim2NonComplex} and \ref{sect:Dim4}, we can see that there is a unique involution
$\varphi$ characterized by any of the following equivalent properties:
\begin{enumerate}
\item[$(i)$] $\varphi$ has a nonzero signature;
\item[$(ii)$] $\varphi$ has the maximal possible signature;
\item[$(iii)$] $X\varphi(X)\in\mathbb{R}_{>0}$ for all nonzero homogeneous $X\in\mathcal{D}$. 
\end{enumerate} 
This distinguished involution appears in items (1-a-1), (1-b-1), (1-c-1), (2-a-3), (2-b-3), (2-c-4), 
(3-a-1), (3-b-1) and (3-c-1).

Let us now turn to the case of Section \ref{sect:Dim2Complex}, that is, $\mathcal{D}\cong M_n(\mathbb{C})$ and 
$\Gamma$ is a division grading on $\mathcal{D}$ as a complex algebra, and consider involutions of the second kind. 
Then there is a unique isomorphism class, (2-f-2-0), of distinguished involutions $\varphi$
characterized by any of the following equivalent properties:
\begin{enumerate}
\item[$(i')$] $\varphi$ has a nonzero signature;
\item[$(ii')$] $\varphi$ has signature $\sqrt{|T_{[2]}|}$;
\item[$(iii')$] for any $t\in T$ of even order $o(t)$, we have that $\varphi(X)=X$ implies $X^{o(t)}\in\mathbb{R}_{>0}$ for all nonzero $X\in\mathcal{D}_t$. 
\end{enumerate} 
Note that the signature of distinguished involutions reaches the maximal possible value, $n$, if and only if 
$T$ is an elementary $2$-group. This latter condition is also necessary and sufficient for the uniqueness of 
a distinguished involution (see Remark \ref{rem:all_inv_2}). Moreover, if it is satisfied, then property $(iii')$ 
is equivalent to property $(iii)$. 

If $T$ is not an elementary $2$-group then the presence of a (fixed) distinguished involution $\varphi$ 
allows us to construct a special basis in the graded subalgebra
\[
\mathcal{D}^{[2]} := \bigoplus_{s\in T^{[2]}}\mathcal{D}_s.
\]
(If $T$ is an elementary $2$-group then $\mathcal{D}^{[2]}=\mathcal{D}_e=\mathbb{C}$.)
The construction is as follows. 

In each component $\mathcal{D}_t$, $t\in T$, we can find a nonzero element $X_t$ 
such that $\varphi(X_t)=X_t$, and this element is determined up to multiplication by a real scalar.
If $o(t)$ is odd, then we can scale $X_t$ so that $X_t^{o(t)}=1$, and this determines the element $X_t$ uniquely.
If $o(t)$ is even, then we can also scale $X_t$ so that $X_t^{o(t)}=1$ because $\varphi$ is distinguished, but 
such an element $X_t$ is unique only up to sign. For $t\in T^{[2]}$, we have a way to choose the sign, which is given 
by the following result.

\begin{lemma}
Fix an isomorphism $\mathcal{D}\cong\mathrm{End}_\mathbb{C}(V)$ and a hermitian form $h$ on $V$ that defines 
$\varphi$, that is, $\varphi=\sigma_h$ as in Equation \eqref{eq:inv_by_h}. 
For any $X\in\mathcal{D}$, set $h_X(v,w) := h(v,Xw)$ for all $v,w\in V$.
Then, for any $s\in T^{[2]}$, we have:
\begin{enumerate}
\item If $o(s)$ is odd, then 
(a) for any $t\in T$, $t^2=s$ implies $X_t^2=X_s$ and 
(b) the signature of $h_{X_s}$ equals the signature of $h$.
\item If $o(s)$ is even, then there exists $\epsilon\in\{\pm 1\}$ such that 
(a) for any $t\in T$, $t^2=s$ implies $X_t^2=\epsilon X_s$ and 
(b) the signature of $h_{X_s}$ equals the signature of $\epsilon h$.
\end{enumerate}
\end{lemma}

\begin{proof}
Suppose $t^2=s$. Since $X_t^2$ belongs to $\mathcal{D}_s$ and satisfies $\varphi(X_t^2)=X_t^2$ and $(X_t^2)^{o(s)}=1$, 
we have $X_t^2=\epsilon X_s$ where $\epsilon=1$ if $o(s)$ is odd and $\epsilon\in\{\pm 1\}$ if $o(s)$ is even.
Next, since $\varphi(X_t)=X_t$, we can write $h_{X_s}(v,w)=h(v,X_s w)=\epsilon h(X_t v, X_t w)$, which shows that
the hermitian forms $ \epsilon h $ and $h_{X_s}$ are isometric.
\end{proof}

Note that, since $\epsilon$ is determined by each of the conditions (a) and (b), it depends only on $X_s$, and 
neither on the choice of $t\in T$ satisfying $t^2=s$ nor on the choice of the isomorphism 
$\mathcal{D}\cong\mathrm{End}_\mathbb{C}(V)$ and hermitian form $h$. If $\epsilon=-1$, we replace $X_s$ by $-X_s$.
We have proved the existence and uniqueness of a basis $\{X_s \mid s\in T^{[2]}\}$ 
for the complex algebra $\mathcal{D}^{[2]}$ 
with the following properties: $\varphi(X_s)=X_s$, $X_s^{o(s)}=1$, and, for any $t\in T$ with $t^2=s$, we have 
that $\varphi(X)=X$ implies $X^2\in\mathbb{R}_{>0}X_s$ for all nonzero $X\in\mathcal{D}_t$.
We will refer to it as the \emph{distinguished basis}.
The following result gives explicit formulas for the products of the elements of the distinguished basis
of $\mathcal{D}^{[2]}$ and also for the quadratic Jordan operators of the basis elements of $\mathcal{D}$ acting on
the distinguished basis of $\mathcal{D}^{[2]}$.

\begin{proposition}
Let $\{X_s \mid s\in T^{[2]}\}$ be the distinguished basis of $\mathcal{D}^{[2]}$. Then:
\begin{enumerate}
\item $X_{u^2}X_{v^2}=\beta(u,v)^2 X_{u^2v^2}$ for all $u,v\in T$.
\item For any $t\in T$,
we have that $\varphi(X)=X$ implies $X X_s X\in\mathbb{R}_{>0}X_{st^2}$ 
for all $ s \in T^{[2]} $ and nonzero $X\in\mathcal{D}_t$.
In particular, if $X$ is scaled to satisfy $X^{o(t)}=1$ then $X X_s X=X_{st^2}$.
\end{enumerate}
\end{proposition}

\begin{proof}
Recall that we chose $X_t$ for all $t\in T$ such that $\varphi(X_t)=X_t$ and $X_t^{o(t)}=1$. 
Then $X_t^2=X_{t^2}$ for all $t\in T$. 

We have $X_u X_v = \lambda X_{uv}$ for some $\lambda\in\mathbb{C}$ with $|\lambda|=1$. 
Applying $\varphi$ to both sides of this equation, we get $X_v X_u=\bar{\lambda}X_{uv}$ and hence $\beta(u,v)=\lambda^2$.
Then, on the one hand, $(X_u X_v)^2=(\lambda X_{uv})^2=\lambda^2 X_{u^2v^2}=\beta(u,v)X_{u^2v^2}$ and,
on the other hand, $(X_u X_v)^2=X_u X_v X_u X_v=\beta(v,u) X_u^2 X_v^2=\beta(v,u)X_{u^2}X_{v^2}$. This proves (1).

For (2), it is necessary and sufficient to prove that $X_t X_s X_t = X_{st^2}$. 
Indeed, pick $u\in T$ such that $u^2=s$ and compute:
\[
X_t X_s X_t = \beta(t,s)X_s X_t^2 = \beta(t,s)X_s X_{t^2} = \beta(t,s)\beta(u,t)^2 X_{st^2}=X_{st^2},
\] 
where in the second last step we have used (1).
\end{proof}

\section*{Acknowledgements and license}

The authors are thankful to Alberto Elduque,
for numerous fruitful discussions about gradings
and mathematics in general,
and to the anonymous referee,
for taking the time to read in detail such a technical text
and suggest significant improvements of exposition.
Adri\'an also wants to thank Alberto
for his support as his thesis supervisor.

Yuri Bahturin was supported by
Discovery Grant 227060-2014
of the Natural Sciences and Engineering Research Council
(NSERC) of Canada.

Mikhail Kochetov was supported by
Discovery Grant 341792-2013
of the Natural Sciences and Engineering Research Council
(NSERC) of Canada
and a grant for visiting scientists by
Instituto Universitario de Matem\'aticas y Aplicaciones,
University of Zaragoza.

Adri\'an Rodrigo-Escudero was supported by
a doctoral grant of the Diputaci\'on General de Arag\'on.
His three-month stay
with Atlantic Algebra Centre
at Memorial University of Newfoundland
was supported by the
Fundaci\'on Bancaria Ibercaja and Fundaci\'on CAI
(reference number CB 4/15),
and by the NSERC of Canada
(Discovery Grant 227060-2014).
He was also supported by
the Spanish Mi\-nis\-te\-rio de Econom\'ia y Competitividad
and Fondo Europeo de De\-sa\-rro\-llo Regional FEDER
(MTM2013-45588-C3-2-P);
and by the Diputaci\'on General de Arag\'on
and Fondo Social Europeo
(Grupo de Investigaci\'on de \'Algebra).

This work is licensed under the
Creative Commons Attribution-NonCommercial-NoDerivatives 4.0 International License.
To view a copy of this license, visit
http://creativecommons.org/licenses/by-nc-nd/4.0/
or send a letter to
Creative Commons, PO Box 1866, Mountain View, CA 94042, USA.

\end{document}